\documentclass[jsl]{psl}
\usepackage{rotate}

\usepackage{color,soul}

\def\confname{}
\def\copyrightyear{2019}

\newcommand{\revision}[1]{#1}

\usepackage{enumerate}
\usepackage{mathrsfs}
\usepackage[foot]{amsaddr}

\newtheorem{thm}{Theorem}

\newtheorem{cor}[thm]{Corollary}

\newtheorem{lem}[thm]{Lemma}

\theoremstyle{definition}
\newtheorem{defn}[thm]{Definition}
\newcommand{\RR}{{\mathbb{R}}}
\newcommand{\R}{{\mathbb{R}}}

\newcommand{\bX}{\mathbf{X}}
\newcommand{\bY}{\mathbf{Y}}
\newcommand{\bZ}{\mathbf{Z}}

\newcommand{\UU}{\mathbb{U}}

\newcommand{\Fraisse}{Fra{\"\i}ss{\'e}\ }
\renewcommand{\hat}{\widehat}

\title{Fra\"iss\'e Limits for Relational Metric Structures}

\author{David Bryant$^1$}
\address{$^1$Department of Mathematics and Statistics, University of Otago, New Zealand. {\tt david.bryant@otago.ac.nz}}
\author{Andr\'{e} Nies$^2$}
\address{$^2$Department of Computer Science, University of Auckland, New Zealand. {\tt   andre@cs.auckland.ac.nz}. {Supported by the Marsden fund of New Zealand.}}
\author{Paul Tupper$^3$}
\address{$^3$Department of Mathematics, Simon Fraser University, Burnaby, British Columbia, Canada. {\tt pft3@sfu.ca}}

\date{\today}                                           % Activate to display a given date or no 
 
\begin{document}

\begin{abstract}
The general theory developed by Ben Yaacov for metric structures provides \Fraisse limits which are approximately ultrahomogeneous. We show here that this result can be strengthened in the case of relational metric structures. We give an extra condition that guarantees exact ultrahomogenous limits. The condition is quite general. We apply it to stochastic processes, the class of diversities, and its subclass of $L_1$ diversities. 
%We develop Ita\"i Ben Yaacov's theory of Fra{\"\i}ss{\'e} limits for metric structures. In classical model theory, \Fraisse limits provide for certain classes of finite structures a universal and ultrahomogenous countable structure, among the best known examples being the Rado graph as the \Fraisse limit of countable graphs, and the linear order of the rationals as the \Fraisse limit of finite linear orders. Ben Yaacov's theory provides an analogous theory of \Fraisse limits for metric structures, which yields the well-known Urysohn space, a universal, ultrahomogeneous Polish metric space, and the Gurarij space, a universal, approximately ultrahomogeneous Banach space. In general, and by necessity, Ben Yaacov's theory only provides approximately ultrahomogeneous limits. Here we show that, with an extra condition, ultrahomogeneous limits are always attained for relational metric structures. We provide three applications of this result:
%\begin{enumerate}
%\item a universal ultrahomogenous stochastic process.
%\item  a separable complete diversity that is analogous to the Urysohn  metric space,
%\item a separable complete ultrahomogeneous $L_1$ diversity that is universal for separable $L_1$ diversities.
%\end{enumerate}
\end{abstract}

\maketitle

\section{Introduction}
The concepts of   homogeneity and universality pervade many areas of mathematics.
 %particularly when the  point of view of mathematical logic is adopted. 
 Both concepts play a central role in the Fra{\"\i}ss{\'e}
 limit~\cite{fraisse1953certaines} of a   class of finite structures with the amalgamation property. For instance,   the Rado (or random)  graph~\cite{rado1964universal}, which   is the Fra{\"\i}ss{\'e}
 limit of the class of undirected finite graphs, is   universal for   the class of countable graphs, and  ultrahomogeneous in the sense  that its isomorphic finite subgraphs are automorphic in the graph. The conjunction of these two properties makes the Rado graph unique up to isomorphism.  This behaviour is entirely typical for   Fra{\"\i}ss{\'e}
 limits.

For structures in the classical sense,  countability is essential  to ensure  this uniqueness. However, we are mainly  interested in  the setting of a complete metric space $X$ with additional structure defined on it. In this context, algebraic embeddings turn into isometric embeddings preserving the structure; countability  turns into  separability, while the spaces themselves  are usually uncountable. The Urysohn metric space $\UU$ is  analogous to the Rado graph; it was  first described by Urysohn~\cite{urysohn1927espace} in 1927,  curiously,   26 years  before the introduction of Fra{\"\i}ss{\'e} limits. The space~$\UU$ is the  completion of the \Fraisse limit of finite metric spaces with rational distances. 
It is determined by being universal for   the separable metric spaces, and  ultrahomogeneous in the sense  that its isometric finite subspaces  are automorphic in \revision{the space}.

Ben Yaacov   \cite{yaacov2008model} has developed a \Fraisse theory for metric structures that is analogous to the classical \Fraisse theory. Under his theory a class of finitely generated structures is confirmed to satisfy some conditions (analogues of the HP, JEP, and AP of the classical theory along with others) and then it is known that the class has a \Fraisse limit.  Under this framework the Urysohn space is the \Fraisse limit of the finite metric spaces, the Urysohn sphere is the limit of finite metric spaces bounded by one, and 
$\ell_2$ is the \Fraisse limit of the finite-dimensional Hilbert spaces. In each of these cases the limit is ultrahomogeneous, but a weaker result is actually established by Ben Yaacov's \Fraisse theory: approximate ultrahomogeneity, meaning that finite partial isomorphism can only be extended to all of the space up to some error. This is necessary for any theory that includes the Guararij space as the limit   
 of the finite-dimensional Banach space, since this limit is only approximately ultrahomogeneous.
 
Here we will show that if we restrict ourselves to {\emph{relational}} metric structures, we are able to add another condition to Ben Yaacov's theory to obtain limits which are exactly ultrahomogeneous, rather than just approximately ultrahomogeneous. The extra condition, which we call the bounded Amalgamation Property (bAP), requires that one-point amalgamations come with a bound on how far apart the two amalgamated points are in the amalgamation. For relational metric structures, Ben Yaacov's conditions along with bAP guarantee strictly ultrahomogeneous \Fraisse limits. 
 
We provide three examples of our result. The first is to construct a universal ultrahomogeneous stochastic process taking values in any finite set. The second is to provide a short proof of the existence and uniqueness of the universal ultrahomogeneous diversity. Diversities were introduced in \cite{Bryant12} as a generalization of metric spaces in which a non-negative value is assigned to all finite sets of points, and not just to pairs. We have established a construction of this diversity by independent means in previous work~\cite{bryant2018}; here it is derived easily from a more general theory.  Finally, we investigate the existence of  universal ultrahomogeneous $L_1$ metrics and diversities. $L_1$ diversities, in analogy with $L_1$ metric spaces, are diversities that can be embedded in the function space $L_1$. Both $L_1$ metrics and diversities are important in applications including combinatorial optimization \cite{Deza97,Bryant14} and phylogenetics \cite{BandeltDress92,BryantKlaere}. We show that there is a universal ultrahomogeneous $L_1$ diversity, which supports the naturalness of the concept of $L_1$ diversities. In contrast, we establish that there cannot be a universal ultrahomogeneous $L_1$ metric space, since the class of finite $L_1$ metric spaces does satisfy the amalgamation property. 
\revision{These} last results appears to be new, despite the fact that $L_1$ metrics have been studied for decades.
%This last result appears to be unknown until now, despite $L_1$ metrics being studied for decades, and 
%The existence of universal ultrahomogeneous $L_1$ diversities supports the naturalness of the new concept.
 
 %(possibly with additional  relations which are closed subsets of $X^n$ for various $n$) 
%$C^*$ algebras and Ben Yaacov. \andre{comment on the Farah paper} \paul{Eagle, Farah, Hart, Kadets}

\section{Fra{\"\i}ss{\'e} limits for metric structures}\label{sec:benyaacov}

%Here we restate the framework and results of Ben Yaacov's theory of Fra{\"\i}ss{\'e} limits for metric structures~\cite{yaacov2008model}. We only state it for relational metric structures, that is, those with only relational symbols in the language. %We avoid all the Apx stuff in the interest of readability. 

%This section is not specific to diversities.

%Finite structures. 

Ben Yaacov's theory \cite{yaacov2008model} concerns  \emph{metric structures}: these are metric spaces with collections of relations (which in continuous logic are real-valued functions of tuples of elements) and functions (which take tuples of elements to other elements). Here we will use Ben Yaacov's theory only for \emph{relational metric structures}, which are metric structures having no functions, and consequently no constants. This restriction has the advantage that finite structures are precisely the same as finitely generated structures.

The following corresponds to \cite[Def.\ 3.1]{yaacov2008model}.
\begin{defn}
Let $\mathcal{L}$ be a collection of symbols (all of which we think of as predicate symbols) each with an associated natural number which is its arity. An \revision{$\mathcal{L}$-\emph{structure}} $\mathfrak{A}$ consists of a complete metric space $(A,d)$ together with, for each $n$-ary predicate symbol $R \in \mathcal{L}$, a uniformly continuous interpretation $R^\mathfrak{A} \colon  A^n \rightarrow \RR$. The symbol $d$ is a distinguished binary predicate symbol. We will follow the convention that Latin letters correspond to the domain of a given structure, e.g. $\mathrm{dom}(\mathfrak{A})=A$. 
\end{defn}

An \emph{embedding} of $\mathcal{L}$-structures $\phi \colon \mathfrak{A} \rightarrow \mathfrak{B}$ is a map  $\phi \revision{\colon A \rightarrow B}$ such that for each $n$-ary predicate symbol $R \in \mathcal{L}$ and all $\bar{a} \in A^n$
\[
R^{\mathfrak{B}}(\phi(\bar{a})) = R^{\mathfrak{A}}(\bar{a}).
\] 
Note that since $d$ is one of the predicate symbols, embeddings are always isometric.
We define $\mathfrak{A}$ to be a \emph{substructure} of $\mathfrak{B}$ if there is an embedding $\phi \colon \mathfrak{A} \rightarrow \mathfrak{B}$. An \emph{isomorphism} is a surjective embedding. An \emph{automorphism} is an isomorphism from a structure to itself.

A \emph{partial isomorphism} $\phi \colon \mathfrak{A} \dashrightarrow \mathfrak{B}$ is an embedding $\phi \colon A_0 \rightarrow B$ where $A _0 \subseteq A$. We say that such a partial isomorphism is \emph{finite} if its domain $A_0$ is finite.

Corresponding to \cite[Def.\ 3.3]{yaacov2008model},
we say a separable structure $\mathfrak{M}$ is \emph{approximately ultrahomogeneous} if every finite partial isomorphism $\phi \colon \mathfrak{M} \dashrightarrow \mathfrak{M}$ is arbitrarily close to the restriction of an automorphism of $\mathfrak{M}$:
for every $\epsilon >0$, there is an automorphism $f$ of $\mathfrak{M}$ such that $d(\phi a, f a) < \epsilon$ for all $a \in \mathrm{dom} \phi$.

%substructure $\mathfrak{B}$ of $\mathfrak{A}$ and embedding $\phi: B \rightarrow A$, there is an isomorphism $\psi$ of $A$ that is within $\epsilon$ of $\phi$ on $B$.
We say a separable structure $\mathfrak{M}$ is \emph{ultrahomogeneous} if every   finite partial isomorphism $\phi \colon \mathfrak{M} \dashrightarrow \mathfrak{M}$ extends to an automorphism.

%substructure $\mathfrak{B}$ of $\mathfrak{A}$ and embedding $\phi: B \rightarrow A$, there is an isomorphism $\psi$ of $A$ that is an extension of $\phi$.

Following \cite[Def.\ 3.5]{yaacov2008model}, we have the following definitions.
\begin{defn}
Let $\mathcal{K}$ be a class of finite structures. 
\begin{enumerate}
\item By a \revision{\emph{$\mathcal{K}$-structure}} we mean an $\mathcal{L}$-structure $\mathfrak{A}$ such that 
all finite substructures of $\mathfrak{A}$ are in $\mathcal{K}$.
%$\Age(\mathfrak{A}) \subseteq \mathcal{K}$
\item We say $\mathcal{K}$ has \revision{\emph{HP (the Hereditary Property)}} if all  members of $\mathcal{K}$ are $\mathcal{K}$-structures.
\item Suppose $\mathcal{K}$ has HP.
We say $\mathcal{K}$ has \revision{\emph{AP (the Amalgamation Property)}} if for every $\mathfrak{B}, \mathfrak{C} \in \mathcal{K}$ and embeddings $f_B \colon \mathfrak{A} \rightarrow \mathfrak{B}$ and $f_C \colon \mathfrak{A} \rightarrow \mathfrak{C}$ from a third finite structure $\mathfrak{A} \in \mathcal{K}$, there is a finite structure $\mathfrak{D} \in \mathcal{K}$ and embeddings $g_B \colon \mathfrak{B}  \rightarrow \mathfrak{D}$, $g_C \colon \mathfrak{C} \rightarrow \mathfrak{D}$ with $g_B \circ f_B = g_C \circ f_C$ on $A$.
\item Suppose $\mathcal{K}$ has HP. The conditions on $\mathcal{K}$ for when \revision{\emph{NAP (the Near Amalgamation Property)}} holds are the same as  for when  AP holds, except we only require that for all $\epsilon>0$ the existence of $g_B, g_C, \mathfrak{D}$ such that 
\[
d(g_B(f_B(a)), g_C(f_C(a))) < \epsilon
\]
for all $a \in A$.
\item We say that $\mathcal{K}$ has \revision{\emph{JEP (the Joint Embedding Property)}} if every two structures in $\mathcal{K}$ embed into a third one.
\end{enumerate}
\end{defn}

Note that for the relational structures we consider here, AP  implies  NAP.

A key part of Ben Yaacov's framework is studying the space of tuples of elements from metric structures as a metric space itself; these are called the enumerated structures.  The metric he defines on the enumerated structures is analogous to the Gromov-Hausdorff distance between compact metric spaces.

%After Yaacov's Def. 3.11:
\begin{defn} \label{def:4}
Let $\mathcal{K}$ be a class of finite structures with NAP. Let $\mathcal{K}_n$ be the \revision{\emph{enumerated structures of length $n$}},  the set of all $n$-element tuples of members of some structures from $\mathcal{K}$.  
\end{defn}

So, $\bar{a} \in \mathcal{K}_n$ means that there is some structure $\mathfrak{A}$ with domain $A$ such that $\bar{a} \in A^n$. The tuple $\bar{a}$ is not a structure in $\mathcal{K}$ itself, because there may be repeated entries, and order is relevant. However,  in a natural way, for each $k$-ary predicate symbol $R$ and each $k$-length sub-tuple of $\bar{a}$, we can apply $R$ to that sub-tuple of $\bar{a}$ using  the interpretation of $R$ in $\mathcal{K}$. Similarly, we can talk of embedding the tuples in $\mathcal{K}_n$ in structures in $\mathcal{K}$. This is any map that takes the entries of $\bar{a} \in \mathcal{K}_n$ and maps them into  a structure
$\mathcal{K}$ while preserving the values of all the relations.

\begin{defn}
%We view such structures as $n$-tuples with relations defined on the elements. 
 Let $\bar a, \bar b \in \mathcal{K}_n$.  We define
 $d_{\mathcal{K}}(\bar a, \bar b)= \inf_{\bar a', \bar b'} \max_i d(\bar a_i', \bar b_i')$ where $\bar a', \bar b'$ are images of  $\bar a ,\bar b$ under embeddings of structures   into a third structure in $\mathcal K$. 
\end{defn}

Informally, for given enumerated structures $\bar a$ and $\bar b$ in $\mathcal{K}_n$ we consider embedding them simultaneously in a common structure in $\mathcal{K}$.  We take the maximum distance in between corresponding $a_i$ and $b_i$ in the embedding for all $i$. Then we take the infimum of this quantity over all such embeddings to get 
$d_{\mathcal{K}}(\bar a, \bar b)$.

Ben Yaacov makes the following observation as a comment, citing his Lemma 3.8.
\begin{lem}
If $\mathcal{K}$ is a class of finite structures satisfying NAP and  JEP 
then
$d_{\mathcal{K}}$ is a pseudometric on $\mathcal{K}_n$ for each $n$.
\end{lem}
\begin{proof}
Use JEP to show that $d_\mathcal{K}(\bar a, \bar b)$ is well-defined and non-negative for all size-$n$ structures $\bar a, \bar b$. Symmetry is straightforward, as is $d_{\mathcal{K}}(\bar a, \bar a)=0$. For the triangle inequality, consider enumerated structures $\bar a, \bar b, \bar c$. Let $\epsilon>0$. Let $(\bar a', \bar b')$ be a joint embedding of $\bar a$ and $\bar b$ so that $\max_i d(\bar a_i', \bar b_i') \leq d_\mathcal{K}(\bar a, \bar b) + \epsilon/3$, and likewise for $(\bar b'', \bar c'')$ so that  $\max_i d(\bar b_i'', \bar c_i'') \leq d_\mathcal{K}(\bar b, \bar c) + \epsilon/3$. Use NAP to get an amalgamation $(\bar a', \bar b', \bar b'', \bar c'')$ where $\max_i d(\bar b', \bar b'') \leq \epsilon/3$, but distances between $\bar a'$ and $\bar b'$, and between $\bar b''$ and $\bar c''$ are preserved. Then for all $i$
\[
d(a_i',c_i'') \leq d(a_i',b_i') + d(b_i',b_i'') + d(b_i'',c_i'') \leq d_\mathcal{K}(a_i,b_i) + d_\mathcal{K}(b_i,c_i) + \epsilon
\]
and so, taking the maximum over all $i$ and letting $\epsilon$ go to zero gives $d_\mathcal{K}(a_i,c_i) \leq d_\mathcal{K}(a_i, b_i) + 
d_\mathcal{K}(b_i,c_i)$.
\end{proof}

 Our next two definitions follow  \cite[Def.\ 3.12]{yaacov2008model}:
\begin{defn}
Let $\mathcal{K}$ be a class of finite structures. 
\begin{enumerate}
\item We say $\mathcal{K}$ satisfies \revision{\emph{PP (the Polish Property)}} if $\mathcal{K}_n$ is separable and complete under $d_{\mathcal{K}}$ for every $n$.
\item We say $\mathcal{K}$ satisfies \revision{\emph{CP (the Continuity Property)}} if for every $n$-ary predicate symbol $P$ the map from $\mathcal{K}_n  \rightarrow \RR$ given by $\bar a \rightarrow P^{\bar a}(\bar a)$ is continuous.
\end{enumerate}
\end{defn}
%
%Ben Yaacov's definition 3.12:
\begin{defn} \label{def:fraisseClass}
Let $\mathcal{K}$ be a class of finite $\mathcal{L}$-structures. We say that $\mathcal{K}$ is a \revision{\emph{Fra{\"\i}ss\'{e} class}} if $\mathcal{K}$ satisfies HP, JEP, AP, PP, CP.
\end{defn}

Our Def.\ \ref{def:fraisseClass} differs from Ben Yaacov's in that we use AP and not NAP, because we want to obtain ultrahomogeneous limits and not just approximately ultrahomogeneous limits (though we will require another, stronger condition).

%Second, we omit the requirement JEP, since for metric structures JEP is implied by AP.
%So for us the difference is we only use AP (not NAP), and we can omit JEP, since AP implies JEP for relational structures.

Ben Yaacov\ does not define  the extension properties explicitly, but \revision{they are} implicit in his Corollary 3.20. We define them here and use them as an alternative definition of a limit of a \Fraisse class.
% because this way we can avoid his ``$\mathrm{Stx}$" notation.

\begin{defn} \label{defn:extprop}
 Let $\mathcal{K}$ be a \Fraisse class and $\mathfrak{M}$ be a separable $\mathcal{K}$ structure.
 \begin{enumerate}
\item
$\mathfrak{M}$ has the \revision{\emph{approximate extension property}} if for all finite $\mathcal{K}$ structures $\mathfrak{B}$, finite enumerated structures $\bar a$ \revision{with elements in $B$}, embedding $h \colon \bar a \rightarrow \mathfrak{M}$ and $\epsilon>0$, there is an embedding $f \colon \mathfrak{B} \rightarrow \mathfrak{M}$  such that $d(f \bar a, h \bar a ) <\epsilon$.  Here $d$ denotes the maximum distance between corresponding elements of the two enumerated structures.
\item
$\mathfrak{M}$ has the \revision{\emph{exact extension property}} if for all finite $\mathcal{K}$ structures $\mathfrak{B}$, finite enumerated structures $\bar a$ \revision{with elements in $B$}, and embedding $h \colon \bar a \rightarrow \mathfrak{M}$, there is an embedding $f \colon \mathfrak{B} \rightarrow \mathfrak{M}$  that extends $h$.
\end{enumerate}
\end{defn}

In the following our approximate limits correspond to Ben Yaacov's limits (his Def.\ 3.15).

\begin{defn}
 Let $\mathcal{K}$ be a \Fraisse class and $\mathfrak{M}$ be a separable $\mathcal{K}$ structure.
 \begin{enumerate}
\item
$\mathfrak{M}$ is an \revision{\emph{approximate limit}} of $\mathcal{K}$ if $\mathfrak{M}$ has the approximate extension property.
\item
$\mathfrak{M}$ is an \revision{\emph{exact limit}} of $\mathcal{K}$ if $\mathfrak{M}$ has the exact extension property.
\end{enumerate}
\end{defn}

Corresponding to Ben Yaacov's Lemma 3.17, Theorem 3.19, and Theorem 3.21 we have the following results.
\begin{lem}
Every \Fraisse class has an approximate limit.
\end{lem}

\begin{thm}
The approximate limit of a \Fraisse class is unique up to isomorphism.
\end{thm}

%Ben Yaacov's Corollary 3.20.
%\begin{cor}
%Let $\mathcal{K}$ be a \Fraisse class, and $\mathfrak{M}$ be a separable $\mathcal{K}$-structure. The following are equivalent:\\
%(i) $\mathfrak{M}$ is the limit of $\mathcal{K}$.\\
%(ii) 
%\end{cor}

\begin{thm}
Let $\mathcal{K}$ be a class of finite relational structures. Then the following are equivalent: \\
(i) $\mathcal{K}$ is a \Fraisse class. \\
(ii) $\mathcal{K}$ is the class of all finite substructures of a separable approximately ultrahomogeneous structure $\mathfrak{M}$. \\
Furthermore, $\mathfrak{M}$ is the approximate limit of $\mathcal{K}$, and is hence unique up to isomorphism and universal for separable $\mathcal{K}$-structures.
\end{thm}

%In the next section we will provide an additional condition on \Fraisse class $\mathcal{K}$ that guarantee the stronger result that $\mathcal{K}$ has an exact limit, and not just an approximate limit.

\section{Approximate versus exact \Fraisse limits} \label{sec:exact}

Here we show that if the $\mathcal{K}$-structures satisfy a property we call  bAP (for  ``bounded Amalgamation Property")  then approximate \Fraisse limits are in fact exact \Fraisse limits. 
%This is what is true of the Urysohn metric space and our Urysohn diversity. 
Recall  from Def.\ \ref{def:4} that for a class $\mathcal{K}$   of finite $\mathcal{L}$-structures  $\mathcal{K}_n$ is the class of enumerated $\mathcal{K}$-structures of length~$n$. 
%Also recall $d_\infty$ from Eqn.\ (\ref{eqn:dinf}).

There is another way to define the distance between two enumerated relational structures that does not appear to have an analogue in Ben Yaacov's paper. We define
\begin{equation} \label{eqn:dinf} 
d_{\infty} (\bar a, \bar b) = \max_{m \leq n} \max_{X \subseteq \{1,\ldots, n\}, |X|=m, R \footnotesize{\mbox{ arity }} m} | R( \bar{a}_X) - R( \bar{b}_X) |,
\end{equation}
where we include the metric $d$ as a binary predicate.
%You look at all $m$,  $1 \leq m \leq n$. 
For each $m$, $1 \leq m \leq n$, we look at all predicates of arity $m$. We then look at all subsets of $\bar a$ of size $m$ and the corresponding subset of $\bar b$ and look at the difference between the predicates on those subsets. We take the $\max$  of the difference over all such $m$, predicates $R$, and index sets $X$.  Unlike with $d_\mathcal{K}$, $d_\infty$ does make use of any common embeddings of $\bar{a}$ and $\bar{b}$ into another metric structure.
%Basically you look at all the relations defined on the $n$-tuples or subsets of them and take the maximum difference of the relation between corresponding subsets.
$d_\infty$  can be thought of as a distance obtained by using the predicates to map enumerated structures into   $\ell^k_\infty$ for some $k$. We will establish the Lipschitz equivalence of $d_{\mathcal{K}}$ and $d_\infty$ for given structures and then use this to prove statements about the space  $(\mathcal{K}_n,d_{\mathcal{K}})$. 

\begin{defn}   We say that $\mathcal{K}$ satisfies \revision{\emph{bAP(the bounded Amalgamation Property)}} if there is a constant $c$ depending only on $n$ such that if $(\bar a, w)$ and $(\bar a,z)$ are two enumerated structures in $\mathcal{K}_n$ with common substructure $\bar a$,  then there is an amalgamation $B=(\bar a, w, z)$ in $\mathcal{K}_{n+1}$ such that 
\[
d_B(w,z) \leq c \, d_\infty((\bar a, w), (\bar a, z)).
\]
%where $d$ indicates the metric in the amalgamation.
\end{defn}

The idea behind bAP is to strengthen AP so that we have some control on how far apart the non-common points are in the amalgamated space. So if   every predicate $R$ takes almost the same value on corresponding subsets of  $(\bar a ,w)$ and $(\bar a,z)$, then $w$ and $z$ are very close to each other in the amalgamation. 

Note that bAP implies AP.  To see this, observe that (forgetting the metric bound) this is a one-point amalgamation result for enumerated structures of the same length. By taking repeated elements if necessary, this is a one-point amalgamation result for finite structures. 
An induction argument gives the   general amalgamation result.

% (see, for example, \paul{Andre, is there some reference for two-point amalgamation implies amalgamation?}). \andre{ I don't know any, this is taken to be a standard argument, just induction}

%:
The next result follows a result for metric spaces originally due to Urysohn; see \cite[Thm.\ 3.4]{Melleray2008a} or    \cite[Thm.\ 1.2.7]{Gao2009}. (We proved a similar result, but in   less generality, in \cite[Lemma 17]{bryant2018}.) Recall Definition~\ref{defn:extprop} where definitions of approximate and exact extension properties are given.
\begin{thm}
Let $\mathcal{K}$ be a \Fraisse class and $\mathfrak M$ be a $\mathcal{K}$-structure. If $\mathcal{K}$ satisfies bAP and $\mathfrak M$ satisfies the approximate extension property, then $ \mathfrak M$ satisfies the exact extension property.
\end{thm}
\begin{proof}
%Let $\mathcal{K}$ be a \Fraisse class satisfying bAP and let  $M$ be an $\mathcal{K}$-structure with the approximate extension property.
 Let  $\bar{a}$ be an enumerated structure of length $n$ with elements taken from $\mathcal{M}$. Let $z$ be a point such that $(\bar a,z) \in \mathcal{K}$ with $\bar a$  embedded in $(\bar a, z)$. It suffices to show that there is a sequence $w_0, w_1, \ldots$ in $\mathfrak M$ such that for all $p$,
\begin{equation} \label{structapproxextenseq}
d(w_p,w_{p+1}) \leq 3 \cdot 2^{-(p+1)} \text{ and }  d_{\infty}((\bar a, z),(\bar a, w_p)) \leq 2^{-p}.
\end{equation}
Because $\mathfrak M$ is complete, the first part of \eqref{structapproxextenseq} shows that $\{w_p\}$ has a limit $w$. The second part of \eqref{structapproxextenseq} shows that $(\bar a, w)$ is isomorphic to $(\bar a,z)$, as required.

Using induction, we will construct the sequence $\{w_p\}$ satisfying conditions \eqref{structapproxextenseq}, along with
structures $M_p=(\bar a, z, w_0, \ldots, w_p)$ which are extensions of both $(\bar a,z)$ and $(\bar a, w_0, \ldots, w_p )$\revision{.}
In particular, the points $\{ w_p \}$ and the structures $M_p$ will satisfy for all $p\geq 0$:
\begin{enumerate}
\item[I.]  $d_{\infty}((\bar a, z),(\bar a, w_p)) \leq 2^{-p}$,
\item[II.] %$M_p$ is a structure with
\begin{enumerate} 
\item $M_p$ is an extension of  $(\bar a,  w_0, \ldots, w_p )$,
\item $M_p$ is an extension of $(\bar a, z)$,
\item  $d(w_p,z) \leq 2^{-p}$ in $M_p$.
\end{enumerate}
\end{enumerate}

First, we use the approximate extension property to get a $w_0 \in M$ such that 
$$d_\infty( (\bar a, w_0),(\bar a, z)) \leq \min( c^{-1}, 1),$$where $c$ is the constant (depending on $n$) in bAP for $\mathcal{K}$. So $w_0$ satisfies condition (I), for $p=0$.
Then using bAP,  there is an amalgamation $M_0=(\bar a, z, w_0)$ in which $d(z, w_0) <1$, thereby satisfying condition (II) for $p=0$.

For the inductive step, suppose, for $p\geq 0$, we have  $w_0,\ldots,w_{p}$ and a  structure $M_p$ satisfying the conditions (I) and (II) above.
 We show  there exist $w_{p+1}$ and $M_p$, satisfying the corresponding conditions for $p+1$.

Condition (II.a) allows us to use  the approximate extension property to get $w_{p+1} \in M$ so that
$$d_\infty( (\bar a, w_0, \ldots w_{p+1}), (\bar a, w_0, \ldots, w_p, z)) \leq \min( c^{-1}, 1) 2^{-(p+1)}.$$ This immediately gives us $d_\infty( (\bar a, w_{p+1}), (\bar a,  z)) \leq  2^{-(p+1)}$ which is condition (I) for $p+1$. bAP allows us to amalgamate $(\bar a, w_0, \ldots w_{p+1})$ and $(\bar a, w_0, \ldots w_{p}, z)$ to get $(\bar a, w_0, \ldots w_{p+1}, z)$. bAP also gives us $d(w_{p+1},z)\leq 2^{-(p+1)}$ in the amalgamation, and so we have all of condition (II) for $p+1$. This concludes the inductive argument.

Now it remains to show that $d(\{w_p, w_{p+1}\}) \leq  3 \cdot 2^{-(p+1)}$ for each $p \geq 1$. In the structure $(\bar a, w_0, \ldots, w_p, w_{p+1},z)$ we have both $d(w_p,z) \leq 2^{-p}$ and $d(w_{p+1},z) \leq 2^{-(p+1)}$, so the triangle inequality gives the result. 
\end{proof}

\begin{thm} \label{thm:ultrahomoglimit}
Let $\mathcal{K}$ be a class of finite relational structures satisfying HP, JEP, bAP, PP, CP. Then $\mathcal{K}$ is a \Fraisse class, and its limit $\mathfrak{M}$ is ultrahomogeneous (therefore an exact limit in our terminology), in addition to being universal for separable $\mathcal{K}$-structures. It is the unique such structure.
\end{thm}
\begin{proof}
bAP implies AP, so $\mathcal{K}$ is a \Fraisse class. Hence $\mathcal{K}$ has a unique approximate limit. By the previous lemma, this approximate limit is in fact an exact limit.
\end{proof}

\section{A Universal Ultrahomogeneous Stochastic Process} \label{sec:stochproc}

%\paul{Style point. Not sure I want to replace all $t_1,\ldots,t_k$ with $t_i$ or $\bar{t}$ yet. think about it.}
%\andre{I think $\bar{t}$ would be good}
In this section we apply our extension of Ben Yaacov's  metric \Fraisse theory to stochastic processes.
We show the existence of a unique universal ultrahomogeneous stochastic process
 %taking value in a finite state space and
  with a separable index set taking values in a finite set.
A key step will be showing how to view stochastic processes as metric structures.

We recall the definition and basic theory of stochastic processes.
Let the  state space $S$ be a finite set and index set $T$ be arbitrary. Let $(\Omega,\mathcal{F},\mathbb{P})$ be a probability space. A \revision{\emph{stochastic process}} is an $S$-valued family of random variables $(X_t)_{t \in T}$ on $\Omega$. It is customary to denote the random variable $X_t$ by $X(t)$.

Every stochastic process has an associated family of finite-dimensional distributions. For every tuple $(t_1,\ldots,t_k)$ of distinct elements, $t_i \in T$ and tuple $(s_1,\ldots,s_k)$, $s_i \in S$ we define
\[
F_{t_1,\ldots,t_k}(s_1,\ldots,s_k) = \mathbb{P}(X(t_1)=s_1, \ldots, X(t_k)=s_k).
\]
In order to simplify notation, in what follows we will denote tuples of values by using an overbar, so that $\bar{t}=(t_1,\ldots,t_k)$ and $\bar{s}=(s_1,\ldots,s_k)$. For example, we would write the above definition as $F_{\bar{t}}(\bar{s}) = \mathbb{P}(X(\bar{t})=\bar{s})$.
%The axioms of probability theory \andre{not sure what these are...} ensure that 
For any stochastic process on $T$ taking values in $S$,
the finite-dimensional distributions satisfy, for any $k$-tuple of distinct values $\bar{t}$ and any $k$-tuple of states $\bar{s}$
% $t_1, \ldots, t_k$
\begin{enumerate}
\item
\( F_{\bar{t}} (\bar{s}) \in [0,1] \)
\item %For all $k\geq 1$, 
\(
  \sum_{\bar{s} \in S^k} F_{\bar{t}} (\bar{s})=1
  \)
 \item $\sum_{z \in S}F_{\bar{t},t_{k+1}}(\bar{s},z) = F_{\bar{t}}(\bar{s})$
 \item For any permutation $\sigma$ of $k$-tuples, $F_{\sigma(\bar{t})}(\sigma(\bar{s}))
=F_{\bar{t}}(\bar{s})$
%\item 
%$F_{t^*, t^*, t_3, \ldots, t_n}(s_1,s_2,s_3, \ldots, s_n)=0$  if $s_1 \neq s_2$.
%\item 
%\(\sum_{s \in S} F_{t_1,t_2}(s,s) < 1 \) for $t_1 \neq t_2$.
\end{enumerate}
The Kolmogorov existence theorem guarantees that any such family of functions $F_{\bar{t}}$ satisfying properties (1), (2), (3), (4) are the finite-dimensional distributions of some stochastic process. Of course, there is not a one-to-one relationship between stochastic processes and families of finite-dimensional distributions: any set of finite-dimensional distributions identifies a whole class of stochastic processes.
In what follows, our theory will work at the level of families of finite-dimensional distributions, or in other words, stochastic processes in the weak sense. However, in order to make arguments more transparent we will often refer to a given stochastic process $(X_t)_{t \in T}$ whose distribution is $F_{\bar{t}}$.

% Accordingly we identify all stochastic processes with the same finite dimensional distributions, and identify the collection with its finite dimensional distribution.
%\paul{In other words, we are just developing a theory of families of finite-dimensional distributions. Or stochastic processes in the weak sense.}

We can  identify each such family of finite-dimensional distributions with a metric structure, once we add a single non-degeneracy condition: We assume that
\begin{enumerate}  \setcounter{enumi}{4}
\item 
\(\sum_{s \in S} F_{t_1,t_2}(s,s) < 1 \) for $t_1 \neq t_2$.
\end{enumerate}
This is equivalent to requiring $\mathbb{P}(X(t_1)= X(t_2))<1$ for $t_1 \neq t_2$. We can view this either as a restriction on the stochastic processes or as identifying the point $t_1$ and $t_2$ in $T$ when $X(t_1)=X(t_2)$ with probability 1. In either case, it ensures that  the index set $T$ is endowed with a metric, as the following lemma shows.
\begin{lem}
%Let $X$ be an $S$-valued stochastic process with index set $T$, satisfying property (5). 
Let $T$ be a set and $F$ a family of finite-dimensional distributions on $T$  for a process taking values in a finite set $S$ that satisfies condition (5).
 For $t_1, t_2 \in T$, define the function 
\[
d(t_1,t_2) = 1- \sum_{s\in S} F_{t_1,t_2}(s,s) %= \mathbb{P}(X(t_1) \neq X(t_2)).
\]
Then $(T,d)$ is a metric space.
\end{lem}

\revision{Note that if $F$ consists of  the finite-dimensional distributions for a stochastic process~$X$, then
\[
d(t_1,t_2) = \mathbb{P}(X(t_1) \neq X(t_2)).
\]}

\begin{proof}
First note that $d(t,t) = 1 - \sum_{s \in S} F_{t,t}(s,s) = 0$ by property (2). Also, property (5) implies that $d(t_1,t_2)>0$ if $t_1 \neq t_2$.

Next, property (4) implies $d(t_1,t_2)=d(t_2,t_1)$. 

To check the triangle inequality, we switch to the stochastic process viewpoint. Note that $X(t_1) \not = X(t_3)$ implies at least one of  $X(t_1) \not = X(t_2)$ and  $X(t_2) \not = X(t_3)$. So
\[
d(t_1,t_3) = \mathbb{P}(X(t_1) \not = X(t_3)) \leq 
\mathbb{P}(X(t_1) \not = X(t_2)) + \mathbb{P} (X(t_2) \not = X(t_3)) = d(t_1,t_2) + d(t_2,t_3).
\]
%\obscure{
%Finally, we check the triangle inequality. Let $t_1,t_2, t_3 \in T$. %It is easiest to use the random process version of the definiton.
%Note that $s_1 \not = s_3$ implies $s_1 \not = s_2$ or $s_2 \not = s_3$.
%\begin{eqnarray*}
%d(t_1,t_3)  & = & \sum_{s_1,s_3 : s_1 \not = s_3} F_{t_1,t_3}(s_1, s_3) \\
%& = & \sum_{s_1,s_2,s_3 : s_1 \not = s_3} F_{t_1,t_2,t_3}(s_1,s_2, s_3) \\
%& \leq & \sum_{s_1,s_2 : s_1 \not = s_2} F_{t_1,t_2}(s_1, s_2) + \sum_{s_2,s_3 : s_2 \not = s_3} F_{t_2,t_3}(s_s, s_3) \\
%&  = & d(t_1,t_2) + d(t_2,t_3).
%\end{eqnarray*}
%}
\end{proof}

To view stochastic processes satisfying property (5)  as relational metric structures, with domain being the index set $T$,
 we only need to define for every $k$-tuple of values from $S$ the predicates 
%\[
%R_{s_1, \ldots, s_k}(t_1, \ldots, t_k) = F_{t_1,\ldots,t_k}(s_1,\ldots,s_k) = \mathbb{P}(X(t_1)=s_1, \ldots, X(t_k)=s_k)
%\]
\[
R_{\bar{s}}(\bar{t}) = F_{\bar{t}}(\bar{s}) = \mathbb{P}(X(\bar{t})=\bar{s})
\]
So the language $\mathcal{L}$ consists of the set of predicate symbols $R_{\bar{s}}$ for any finite tuple of values $\bar{s}$, along with metric $d$, which we always interpret by $d(t_1,t_2) = 1 - \sum_{s \in S} R_{s,s}(t_1,t_2)$.
%(Strictly speaking, finite dimensional distributions $F$ are only defined for distinct $t_1, \ldots, t_k$, but here we make the natural extension to all tuples using the alternative definition in terms of $\mathbb{P}$ and $X$.)

We need to show that each of these predicates is uniformly continuous. 
\begin{lem} \label{lem:stochLipschitz}
For each $k$-tuple $\bar{s}$, the relation $R_{\bar{s}}$ is $1$-Lipschitz   in each of its arguments.
\end{lem}
\begin{proof}
Without loss of generality, consider the first argument. For $t_1, t_1^* \in T$ we have
\begin{eqnarray*}
R_{s_1,\ldots,s_k}(t_1,\ldots,t_k) 
& = & 
 \mathbb{P}( X(t_1)=s_1, \ldots, X(t_k)=s_k) \\
& \leq &  \mathbb{P}( X(t_1^*)=s_1, \ldots, X(t_k)=s_k) + \mathbb{P}(X(t_1^*) \neq X(t_1)) \\
& = & R_{s_1,\ldots,s_k}(t_1^*, \ldots, t_k) + d(t_1,t_1^*),
\end{eqnarray*}
where we have used the fact that $X(t_1)=s_1$ can only happen if at least one of $X(t^*_1)=\revision{s_1}$ or $X(t_1^*) \not = X(t_1)$ occurs.
Switching $t_1$ and $t_1^*$ gives 
\[
R_{s_1,\ldots,s_k}(t_1^*, \ldots, t_k) \leq R_{s_1,\ldots,s_k}(t_1,\ldots,t_k) + d(t_1,t_1^*),
\]
and hence 
\[
|R_{s_1,\ldots,s_k}(t_1^*, \ldots, t_k) - R_{s_1,\ldots,s_k}(t_1,\ldots,t_k)| \leq d(t_1,t_1^*),
\]
\end{proof}

Now we formally define the class of metric structures $\mathcal{SP}$.
\begin{defn} \label{def:SP}
A \revision{\emph{separable non-degenerate $S$-valued stochastic process}} is a separable metric space $(T,d)$ with $k$-ary predicates $R_{\bar{s}}$ for $k\geq 2$, $\bar{s} \in S^k$ such that, if we let $F_{\bar{t}}(\bar{s}) = R_{\bar{s}}(\bar{t})$ for all $k, \bar{s}\in S^k, \bar{t}\in T^k$, then conditions (1) through (4) as well as
\begin{equation*}
(5') \hspace{.5cm} 1- \sum_{s\in S} F_{t_1,t_2}(s,s)= d(t_1,t_2) 
\end{equation*}
hold. We denote the set of all such metric structures by $\mathcal{SP}$.
\end{defn}

We will not actually use the notation $R_{\bar{s}}$  in what follows. Rather, we will stay with  either $F_{\bar{t}}(\bar{s})$ or $X(\bar{t})$ to keep in line with probabilistic notation.
  
 To summarize, any metric structure satisfying Definition~\ref{def:SP} gives a family of finite-dimensional distributions corresponding to a non-degenerate stochastic process on $T$ taking values in $S$. On the other hand, any such non-degenerate stochastic process corresponds to a metric structure of the type in Definition~\ref{def:SP}.
%Relational metric structures are metric spaces with additional predicates defined on them.  So we start with a complete metric space $(T,d)$. For each tuple of elements from $S$ we define predicates $R_{s_1, \ldots, s_k}(t_1, \ldots, t_k)$. This forms our signature. The axioms we require a stochastic process to satisfy now correspond to the four properties above, plus conditions (5') which is
%\begin{equation*}
%(5') \hspace{.5cm} 1- \sum_{s\in S} F_{t_1,t_2}(s,s)= d(t_1,t_2) 
%\end{equation*}
%for all $t_1,t_2 \in T$.
%This makes it possible for us to have a connection to the original theory of stochastic processes.
This metric structure view of non-degenerate stochastic processes immediately gives a definition of completion and separability for stochastic processes, namely, whether the space $(T,d)$ with the induced metric $d$ is complete and separable, respectively. 
%This latter definition does  not correspond exactly with the definition of separability for   stochastic processes,  but they are closely related, as we will discuss later.
%\paul{I will have to explain better here.}

%The  metric language $\mathcal{L}$ is given by   finitely  many  $k$-ary predicate symbols  $R_{s_1,\ldots,s_k}$ for each $k$, together with the binary relation symbol~$d$.  Let $\mathcal{K}$ be the class of all   finite $\mathcal{L}$-structures satisfying properties (1)-(4) and (5'). 

Let  $\mathcal{K}$ be the class of all finite metric structures in $\mathcal{SP}$.
We will show that $\mathcal{K}$ is a \Fraisse class satisfying bAP.
Note that for any two $n$-tuples in $\mathcal{K}_n$, we have 
%\andre{is this right? Where did $S$ go?}\paul{hope it's fixed now}
\begin{eqnarray*}
d_\infty(\bar a, \bar b)& = & \max_{|A|\leq n, \bar{s} \in S^{|A|}} | R_{\bar{s}}(\bar{a}_A) - R_{\bar{s}}(\bar{b}_A) |
\\
& = &  \max_{|A|\leq n, \bar{s} \in S^{|A|}} |\mathbb{P}(X(\bar{a}_A)=\bar{s}) - \mathbb{P}(X(\bar{b}_A)=\bar{s})|
\end{eqnarray*}
where $A$ runs over all subsets of $\{1, \ldots, n\}$.

Here is our main result for this section.

\begin{thm} \label{thm:stocprocFraisse}
The class $\mathcal{K}$ of all finite non-degenerate stochastic processes satisfies HP, JEP, bAP, PP, and CP, and therefore has a unique \Fraisse limit.
\end{thm}

A key component of the proof of this theorem is the Coupling Lemma of probability theory (see, for example \cite[p.19]{lindvall92}): between any two random variables $X$ and $Y$ taking values in some state space \revision{$U$}, we can define the total variation distance between their distributions as 
\[
d_{TV}(X,Y) = \sup_{A \subseteq \revision{U}} | \mathbb{P}(X \in A) - \mathbb{P}(Y \in A)| = \frac{1}{2} \sum_{\revision{u \in U}}
\left| \mathbb{P}(X = \revision{u}) - \mathbb{P}(Y=\revision{u})\right| .
\]
\revision{The Coupling Lemma asserts the following:  given any two $U$-valued random variables $X$ and $Y$, there are    random variables $\tilde{X}$ and $\tilde{Y}$ defined on the same probability space   so that $\tilde{X}$ has the same distribution as $X$, $\tilde{Y}$ has the same distribution as $Y$, and} 
%The Coupling Lemma asserts that new jointly defined random variables $\tilde{X}$ and $\tilde{Y}$ can give so that $\tilde{X}$ has the same distribution as $X$, $\tilde{Y}$ has the same distribution as $Y$, and 
\[
\mathbb{P}(\tilde{X} \not = \tilde{Y}) = d_{TV}(X,Y).
\]
\revision{In what follows, we will use the Coupling Lemma with $X$ and $Y$ being tuples of random variables of length $n$, and $U$ being $S^n$.
}

%
%\begin{lem} {\bf (Coupling Lemma.)} \cite{gibbs2002,lindvall92}
%Let $(\bar{a},d_1)$ and $(\bar{b},d_2)$ be two members of $\mathcal{K}_{n+1}$ such that 
%
%Given any two random variables $X$, $Y$ with the same state space, the random variables can be jointly defined on a probability 
%\[
%\mathbb{P}(X \neq Y) = d_{TV}(X,Y).
%\]
%\end{lem}
%\begin{proof} 
%See, for example, \cite[p. ]{lindvall92}.
%\end{proof}
%

\emph{Proof of Theorem \ref{thm:stocprocFraisse}}

HP. This follows immediately.

JEP.  This follows from bAP if we identify a single point in each of the two index sets of the stochastic processes we want to jointly embed.
%Given two finite non-degenerate stochastic processes $X$ and $Y$, 
%%possibly defined on different probability spaces, 
%let $\widetilde{X}$ be a process with the same distribution as $X$ and $\widetilde{Y}$  a process with the same distribution as $Y$ such that $\widetilde{X}$ and $\widetilde{Y}$ are defined on the same probability space and are independent. If necessary, identify points in the union of the index sets so that $(\widetilde{X},\widetilde{Y})$ is non-degenerate.

bAP. Let $X=(\bar a,w)$ and $Y=(\bar a, z)$ be two enumerated structures in $\mathcal{K}_{n+1}$, where $X$ and $Y$ agree on $\bar a$. 
%Let $X$ be the stochastic process defined either on $(\bar a, w)$ or $(\bar a,z)$. 
We wish to amalgamate $X$ and $Y$ to obtain $Z=(\bar a, w, z)$. 

To do this, we shift to random variable notation.
$X$ and $Y$ correspond to two random pairs $(\bX_{a},\bX_{w})$ and $(\bY_{a},\bY_z)$ such that $\bX_a$ and $\bY_a$ (taking values in $S^n$) have the same distribution and $\bX_w$ and $\bY_z$ take values in $S$. We will create a random variable $(\bZ_a,\bZ_w,\bZ_z)$ such that  $(\bZ_a,\bZ_w)$ has the same distribution as $(\bX_{a},\bX_{w})$, $(\bZ_a,\bZ_z)$ has the same distribution as $(\bY_{a},\bY_z)$, and we can bound the probability that $\bZ_w$ and $\bZ_z$ are different.
%$\Pr[\bX_a = \bY_a] = 1$. 

Fix $\overline{s} \in S^n$. We let $\bX_{a|\overline{s}}$ denote the random variable $\bX_w$ conditioned on $\bX_a = \overline{s}$; this is the random variable defined by the distribution
\[ \mathbb{P}[\bX_{a|\overline{s}}= s ] = \frac{\mathbb{P}[(\bX_a , \bX_w) = (\overline{s},s) ]} { \mathbb{P}[\bX_a = \overline{s}]}.\]
Let $\bY_{a|\overline{s}}$ be defined similarly. 

For each value of $\overline{s}$ we apply the Coupling Lemma to $\bX_{a|\overline{s}}$ and $\bY_{a|\overline{s}}$ to obtain the jointly defined random variables $\bZ^X_{a|\overline{s}}$ and $\bZ^Y_{a|\overline{s}}$.  $\bZ^X_{a|\overline{s}}$ has the same distribution as $\bX_{a|\overline{s}}$, $\bZ^Y_{a|\overline{s}}$ has the same distribution as $\bY_{a|\overline{s}}$, and 
\[
\mathbb{P}\left( \bZ^X_{a|\overline{s}} \neq \bZ^Y_{a|\overline{s}}\right)= d_{TV}( \bX_{a|\overline{s}}, \bY_{a|\overline{s}}).
\]

We can assemble $(\bZ_a,\bZ_w,\bZ_z)$  as follows. For each $\overline{s}\in S^n$ define 
\[
\mathbb{P} [ (\bZ_a,\bZ_w,\bZ_z)=(\overline{s},s_1,s_2)]= 
\mathbb{P} [\bX_a= \overline{s}] \mathbb{P} [ \bZ^X_{a|\overline{s}} = s_1, \bZ^Y_{a|\overline{s}}=s_2].
\]

We let the tuple $(a,w,z)$ correspond to the random variable $(\bZ_a,\bZ_w,\bZ_z)$.
It remains to show that we can bound $d(w,z)$. First we re-express the quantity in terms of $\bZ^X_{a|\overline{s}}$ and $\bZ^Y_{a|\overline{s}}$ by conditioning on $\bZ_a$.
\begin{eqnarray*}
d(w,z) & = & \mathbb{P}(\bZ_w \neq \bZ_z)  \\
& = & \sum_{\overline{s}} \mathbb{P}[ \bZ_a = \overline{s} ] \mathbb{P} [
\bZ_w \neq \bZ_z | \bZ_a = \overline{s} ] \\
& = &\sum_{\overline{s}}  \mathbb{P}[ \bZ_a = \overline{s} ] \mathbb{P} [ \bZ^X_{a|\overline{s}} \neq 
\bZ^Y_{a|\overline{s}}].
\end{eqnarray*}
By our use of the Coupling Lemma we then have
\begin{eqnarray*}
d(w,z) & = & \sum_{\overline{s}} \mathbb{P}[ \bX_a = \overline{s} ] d_{TV}( \bX_{a|\overline{s}}, \bY_{a|\overline{s}}) \\
& = & \sum_{\overline{s}} \mathbb{P}[ \bX_a = \overline{s}  ] \frac{1}{2} \sum_{s \in S}  \left|
\mathbb{P}( \bX_{a|\overline{s}}=s ) - \mathbb{P}( \bY_{a|\overline{s}} = s) \right| \\
& \leq &\frac{1}{2} \sum_{\overline{s}} \sum_{s \in S}  \left|
\mathbb{P}[ (\bX_{a},\bX_w)=(\overline{s},s) ] - \mathbb{P}[ (\bY_{a}, \bY_z) =(\overline{s},s)] \right| \\
& \leq & \frac{1}{2} |S|^{n+1} d_\infty( (a,w), (a,z)),
\end{eqnarray*}
as required.

PP. Note that, for each $n$, the conditions for a non-degenerate stochastic process define a closed subspace of $\mathbb{R}^k$ for a sufficiently large $k$.   As in the case of diversities it suffices to show that $d_{\mathcal{K}}$ and $d_\infty$ are Lipschitz equivalent, since $d_\infty$ is Lipschitz equivalent to the Euclidean metric. We will show, for any pair of $n$-tuples $\bar a, \bar b$ that
\[
 \frac{1}{n}  d_\infty(\bar a, \bar b) \leq d_\mathcal{K}(\bar a, \bar b) \leq  |S|^n d_\infty(\bar a, \bar b)
\]

To prove the first inequality, 
%\paul{might want to just observe that it is a consequence of Lipschitz continuity of predicates} 
suppose that we have $\bar a$ and $\bar b$ embedded together in $(\bar a, \bar b)$ so that 
$\max_{i} \mathbb{P}(X(a_i) \neq X(b_i) ) \leq d_\mathcal{K}(\bar a, \bar b) + \epsilon$ for some $\epsilon>0$.
For any subset $X$ of $\{1, \ldots, n\}$ and any element $\alpha$ of $S^{|X|}$ we have 
\begin{eqnarray*}
| \mathbb{P}(X(a_X) =\alpha) - \mathbb{P} (X(b_X) =\alpha )| &\leq& \mathbb{P}( X(a_X) \neq X(b_X)) \leq \mathbb{P}( X(\bar a) \neq X(\bar b)) \\
& \leq & n \max_i \mathbb{P}(X(a_i) \neq X(b_i) ) \leq n d_\mathcal{K}(\bar a, \bar b) + n \epsilon.
\end{eqnarray*}
Dividing by $n$ and letting $\epsilon$ go to zero gives the first inequality.

% To prove the second inequality,  let $d_\infty(\bar a, \bar b)= \epsilon$. We will construct a joint embedding of $\bar a$ and $\bar b$
% where 
% 
% let $(\bar a, \bar b)$  
% 

 To prove the second inequality, note that for any $A \subseteq S^n$, we have 
 \[
 |\mathbb{P}(X(\bar a)\in A) - \mathbb{P}(X(\bar b)\in A)| \leq \sum_{\alpha \in A}  |\mathbb{P}(X(\bar a)=\alpha) - \mathbb{P}(X(\bar b)=\alpha)|    \leq |S|^n d_\infty(\bar a, \bar b).
\]
\revision{
So we obtain the following  from the definition of total variation distance:
\[
d_{TV}(X(\bar a),X(\bar b)) \leq |S|^n d_\infty(\bar a, \bar b).
\]
}
 Using the \revision{Coupling Lemma} on the whole random vectors $X(\bar a)$ and $X(\bar b)$, we can find a joint embedding so that
 $\mathbb{P}(X(\bar a) \neq X(\bar b)) \revision{ \leq d_{TV}(X(\bar a),X(\bar b))} \leq |S|^n d_\infty(\bar a, \bar b)$.
Now we have that 
\[
\max_i \mathbb{P}( X(a_i) \neq X(b_i)) \leq \mathbb{P}(X(\bar a) \neq X(\bar b)) \leq |S|^n d_\infty(\bar a, \bar b)
\]
and the inequality is established.

CP. The $n$-ary predicates on $\mathcal{K}$ are, for tuples $\alpha \in S^n$, 
\[
R_\alpha(t_1, \ldots, t_n) = \mathbb{P}(X(t_1)= \alpha_1, \ldots, X(t_n)=\alpha_n).
\]
Let $\bar a$ and $\bar b$ be two structures of length $n$ such that $d_\mathcal{K}(\bar a, \bar b)=\epsilon$. Let them be jointly embedded in $(\bar a, \bar b)$ such that $\max_i d(a_i,b_i) \leq 2 \epsilon$. Then Lemma~\ref{lem:stochLipschitz} implies that
\[
|R_\alpha(\bar a) - R_\alpha(\bar b)| \leq 2 n \epsilon.
\]
as required.
\qed
\\
%\end{proof}

%\paul{What follows seems pretty redundant, I know.} \andre{Too many "We say" rather use italics for the term. Maybe shorten a bit, too?}

Now we will translate the existence of a \Fraisse limit for the class $\mathcal{K}$ into the language of stochastic processes. We will need to introduce some  terminology that is non-standard in the context of stochastic processes.
Let $X_t$ for $t \in T$ and $Y_u$ for $u \in U$ be two $S$-valued stochastic processes. Let $\phi \colon T \rightarrow U$ be given. The map $\phi$ \emph{embeds} $(T,X)$ in $(U,Y)$ if the finite-dimensional distributions of $Y_{\phi(t)}, t \in T$ are identical to those of $X_t, t \in T$.  Furthermore, $\phi$ is an \emph{isomorphism} if $\phi$ is onto. An \emph{automorphism} is an isomorphism from a stochastic process to itself.
A stochastic process is \emph{non-degenerate} if $d(t_1,t_2) := \mathbb{P}(X(t_1) \not = X(t_2) )>0$ is non-zero for all $t_1 \not = t_2$.  A stochastic process is \emph{separable} if $(T,d)$ is a separable metric space. (This conflicts with the notion of separability in the stochastic processes literature.)
A stochastic process $X_t, t \in T$  is \emph{finite} if $T$ is finite.

A stochastic process is \emph{universal} if any finite \revision{$S$-valued} stochastic process can be embedded in it.
A map $\phi$ is a \emph{partial isomorphism} from $X_t, t \in T$ to $Y_u, u \in U$ if $\phi \colon T_0 \rightarrow U$ is an embedding for some $A_0 \subseteq T$. We say $\phi$ is a finite partial isomorphism if $A_0$ is finite.
A stochastic processes is \emph{ultrahomogeneous} if any finite partial isomorphism from the process to itself can be extended to an automorphism.
  
\begin{cor}
There is a separable universal ultrahomogeneous process $\mathbb{X}_t, t \in \mathbb{\mathbb{T}}$ that is unique up to isomorphism. 
\end{cor}

%{\bf Examples of separable processes}
%
%{\bf Example 1 } Telegraph process. This is also known as the two-state continuous time Markov process. We let $S={-1,1}$ or any other pair of states, $T= \mathbb{R}$. The stationary version of the process is that
%
%{\bf Example 2 } ACGT on a Tree. Let $S=\{ A,C,G,T \}$. Let $T$ be a real tree. Given any continuous time stochastic p
%
%{\bf Example 3 } Random Disk Packing in the Plane. Let $T= \mathbb{R}^2$. Now imagine placing discs of radius 1 with centres according to a Poisson process with rate $\lambda$. Once the discs are placed, every point in the plane covered by a disc is assigned value 1, and every other point gets value 0. 

\section{Diversities} 

%An (undirected) $k$-hypergraph is a set with a specified collection of $k$-element subsets, and an $A$-hypergraph for $A \subset \mathbb N$ is a set that is a $k$-hypergraph for each $k\in A$. The collection of finite $A$-hypergraphs permits a \Fraisse limit. These structures have been studied occasionally in the literature; for instance see \cite{macpherson2011survey} for some results about $k$-hypergraphs. The metric analog of an $\NN$-hypergraph (or just \emph{hypergraph} as it is known in discrete mathematics) is the following concept  introduced in \cite{Bryant12}.
%

A {\em diversity} \cite{Bryant12} is a pair $(X,\delta)$ where $X$ is a set and $\delta$ is a function from the finite subsets of $X$ to $\mathbb{R}$ satisfying  
\begin{quotation}
\noindent (D1) $\delta(A) \geq 0$, and $\delta(A) = 0$ if 
and only if 
$|A|\leq 1$. \\
(D2) If $B \neq \emptyset$ then $\delta(A\cup B) + \delta(B \cup C) \geq \delta(A \cup C)$
\end{quotation}
for all finite $A, B, C \subseteq X$.  
% getridofpseudo
%We say that $(X,\delta)$ is a pseudo-diversity if it satisfies (D2) together with
%\begin{quotation}
%\noindent (D1') $\delta(A) \geq 0$, and $\delta(A) = 0$ if 
%$|A|\leq 1$. 
%\end{quotation}
Diversities     form an extension of the  concept of a metric space. \revision{Property (D2) is the \emph{triangle inequality} and is the analogue of the triangle inequality for metric spaces.} Every diversity has an {\em induced metric}, given by $d(a,b) = \delta(\{a,b\})$ for all $a,b \in X$. Note also that  $\delta$ is {\em monotonic}: $A \subseteq B$ implies $\delta(A) \leq \delta(B)$. Also $\delta$ is {\em subadditive on sets with nonempty intersection}: $\delta(A \cup B) \leq \delta(A) + \delta(B)$ when $A \cap B \neq \emptyset$ \cite[Prop. 2.1]{Bryant12}. \revision{Monotonicity and subadditivity on overlapping sets are also sufficient to establish the triangle inequality: If $B\neq \emptyset$,
\begin{equation}\label{eqn:prooftri}
\delta(A \cup C ) \leq \delta( A \cup B \cup C) \leq \delta(A \cup B) + \delta(B \cup C).
\end{equation}}
%For convenience, in the remainder of the paper we will relax condition (D1) and allow $\delta(A) = 0$ even when $|A|>1$.  Likewise, for metrics we allow $d(x,y)=0$ even if $x \neq y$.
\revision{Just as a semimetric generalizes a metric space by allowing $d(x,y)=0$ for $x \neq y$, we define a \emph{semidiversity} to be a pair $(X,\delta)$ that satisfy (D1) and (D2) above except that we may have $\delta(A)=0$ for $|A|>1$.}

%We say that a diversity $(X,\delta)$ is \emph{complete} if its induced metric $(X,d)$ is complete \cite{Poelstra13},  and that a diversity is \emph{separable} if its induced metric is separable.

In \cite{bryant2018} we constructed the diversity analog  $(\mathbb{U}, \delta_{\mathbb{U}})$ of the Urysohn metric space. It is determined uniquely by being universal for separable diversities, and  ultrahomogeneous in the sense that isomorphic finite subdiversities are  automorphic. We also showed that the induced metric space of $(\mathbb{U}, \delta_{\mathbb{U}})$ is the Urysohn metric space. Our method of constructing this Urysohn diversity in \cite{bryant2018} was analogous to Kat{\v{e}}tov's  construction of the Urysohn metric space \cite{katetov1986universal}. 

Here we demonstrate the existence of $(\mathbb{U}, \delta_{\mathbb{U}})$ via Ben Yaacov's theory of metric  \Fraisse limits \cite{yaacov2015fraisse}, as we described in Section~\ref{sec:benyaacov}, along with our results in Section~\ref{sec:exact} is order to prove that the \Fraisse limit is ultrahomogeneous and not just approximately ultrahomogeneous.

%
%We're motivated to study Ben Yaacov's theory of metric  \Fraisse limits \cite{yaacov2015fraisse} to show the existence of a Urysohn diversity.
%We start with Ben Yaacov's theory and specialize it to the case of relational metric structures. Ben Yaacov's theory only guarantees the existence of approximately ultrahomogeneous \Fraisse limits. We provide conditions under which approximate ultrahomogeneous structures are in fact ultrahomogeneous in the strong sense. We use these results to prove the existence of an ultrahomogeneous, universal separable complete diversity.
%

%Diversities were introduced in \cite{Bryant12}.

In order to study diversities using Ben Yaacov's theory, it is necessary to describe them as metric structures. Every diversity immediately has a metric space associated with it, the induced metric.  The diversity function $\delta$ takes a variable number of distinct arguments, so this does not immediately fit into model theory. Hence we define $\delta_k$ as a predicate for $k \geq 1$. We define $\delta_k$ of a tuple in $X^k$ to be $\delta$ of the set of values that the tuple takes. For example,
\[
\delta_3(x,y,x)= \delta(\{x,y\})
\]
for all $x,y \in X$.
In this way, diversities are relational metric structures with a countable number of predicates, one with arity $m$ for every $m\geq 2$.

What axioms do diversities satisfy as metric structures?  In the following all tuples are assumed to be of non-zero length. The set function $\delta$ being a diversity is equivalent to 
\begin{enumerate}
\item $\delta_1(x)=0$ for all points $x \in X$.
%\item \revision{$\delta_2(x,y)=d(x,y)$ for all $x, y \in X$.}
\item $\delta_k$ is permutation invariant.
\item \revision{ If $\delta_k(a_1,\ldots,a_k)=0$ then $a_1=\cdots=a_k$.}
\item For all $a_1,\ldots, a_k \in X$,
\(
\delta_{k+1}( a_1, \ldots, a_{k-1}, a_k, a_k ) = \delta_k( a_1, \ldots, a_{k-1}, a_k).
\)
\item $\delta_{k+1}(\bar{a},b) \geq \delta_k(\bar{a})$ for all $\bar a \in X^k$, $b \in X$.
\item For all tuples $\bar{a}, \bar{b}, \bar{c}$ with lengths $j,k,\ell$, if $k \geq 1$
\[
\delta_{j+\ell}( \bar{a}, \bar{c}) \leq \delta_{j+k}(\bar{a},\bar{b}) + \delta_{k+\ell}(\bar{b},\bar{c}).
\]
\end{enumerate}
Note that these conditions imply that $d \equiv \delta_2$ is a metric on $X$.
To be a metric structure, we need to confirm that each predicate is uniformly  continuous, as  shown in the following lemma.

\begin{lem} \label{lem:1-Lipschitz} Let $(X, \delta)$ be a diversity. For each $n$, the function $\delta^{(k)}$ is 1-Lipschitz in each argument. \end{lem}
\begin{proof}
Consider varying the $i$th argument of $\delta^{(k)}$ from $x_i$ to $x_i'$. We know from the triangle inequality that 
\begin{eqnarray*}
\delta^{(k)}(x_1, \ldots, x_i, \ldots, x_k) & = & \delta(\{x_1, \ldots, x_i, \ldots, x_k\}) \\
 & \leq & \delta(\{x_1, \ldots, x_i', \ldots, x_k\}) + \delta(\{ x_i, x_i'\}) \\
 & = & \delta^{(k)}(x_1, \ldots, x_i', \ldots, x_k) + d(x_i,x_i').
\end{eqnarray*}
Similarly, $\delta^{(k)}(x_1, \ldots, x_i', \ldots, x_k) \leq \delta^{(k)}(x_1, \ldots, x_i, \ldots, x_k) +
d(x_i,x_i')$. So 
\[
|  \delta^{(k)}(x_1, \ldots, x_i, \ldots, x_k) - \delta^{(k)}(x_1, \ldots, x_i', \ldots, x_k)| \leq d(x_i, x_i')
\]
as required.
\end{proof}

Since these are closed conditions, for any finite set $X$ with $|X|=n$, the set of diversities can be viewed as a closed subset of $\mathbb{R}^{k}$ for $k = n + n^2 + \cdots + n^n$.

Diversities have already been given a definition of completeness and separability in \cite{Poelstra13}: a diversity is complete if its induced metric space is complete and it is separable if its induced metric space is separable.  Fortunately, these correspond precisely with the definitions one would get   with viewing diversities as metric structures, so the \Fraisse limit will give us exactly what we want.

We define a special one-point amalgamation of diversities that yields the bounded Amalgamation Property (bAP).

\begin{defn} \label{defn:ma}
Let $(X \cup \{z_1\},\delta)$ and $(X \cup \{z_2\},\delta)$ be two diversities that agree on $X$.  The amalgamation of the diversities is defined on $X \cup \{z_1,z_2\}$ by
\begin{multline}
\delta(A \cup \{z_1, z_2\}) =  \\ \max \left[
\sup_{B \subseteq X} \delta( A \cup B \cup \{z_1\} ) - \delta( B \cup \{z_2\}) , \sup_{C \subseteq X} \delta( A \cup C \cup \{z_2\}) - \delta(C \cup \{z_1\})
\right]
%\sup_{B, C \subseteq X} \delta( A \cup B \cup C) - \delta( B \cup \{z_1\}) - \delta( C \cup \{z_2\}).
\end{multline}
\end{defn}

The idea of this definition of the amalgamation is to obtain the minimal diversity that extends both the diversities on $X \cup \{z_1\}$ and $X \cup \{z_2\}$. To see this,
note that however we define the amalgamation, the triangle inequality requires
\[
\delta(A \cup \{z_1, z_2\}) \geq  \delta( A \cup B \cup \{z_1\} ) - \delta( B \cup \{z_2\})
\]
for all $B  \subseteq X$ and 
\[
\delta(A \cup \{z_1, z_2\}) \geq \delta( A \cup C \cup \{z_2\}) - \delta(C \cup \{z_1\})
\]
for all $C \subseteq X$. So the value of $\delta(A \cup \{z_1,z_2\})$ can not be less than the one we have defined.

%We first give some alternative formulas for the metric amalgamation.
%\begin{lem}
%For any diversity $(X,\delta)$, and set $A \subseteq X$, and point $z \in X$
%\[
%\delta( A \cup \{z\}) = \sup_{B \subseteq X} \delta( A \cup B) - \delta( B \cup \{z\}).
%\]
%\end{lem}
%\begin{proof}
%By the triangle inequality for diversities $\delta( A \cup \{z\}) \geq \delta(A \cup B) - \delta( B \cup \{z\})$ for all finite $B \subseteq X$. By letting $B = \{z\}$ we see that this upper bound is attained.
%\end{proof}
%
%This lemma gives two alternative formulae for the metric amalgamation:
%\begin{eqnarray} 
%\delta(A \cup \{z_1, z_2\}) & = & \sup_{B \subseteq X} \delta( A \cup B \cup \{z_1\}) - \delta( B \cup \{z_2\}) \label{eqn:firstmadef} \\
%\delta(A \cup \{z_1, z_2\}) & = & \sup_{C \subseteq X} \delta( A \cup C \cup \{z_2\}) - \delta( C \cup \{z_1\}) \label{eqn:secondmadef}
%\end{eqnarray}
%\andre{this is a lemma}
\begin{lem}
The amalgamation given by Def.\ \ref{defn:ma} is a diversity on $X \cup \{z_1,z_2\}$.
\end{lem}
\begin{proof}
To establish that $\delta$ is a diversity, it suffices to prove  
\begin{enumerate}[(a)]
\item monotonicity: $\delta(A) \leq \delta(B)$ when $A \subseteq B$, and 
\item subadditivity on overlapping sets: $A \cap B \not = \emptyset$ implies $\delta(A \cup B) \leq \delta(A) + \delta(B)$,
\end{enumerate}
\revision{see \eqref{eqn:prooftri}}.

Monotonicity with respect to $A$ follows easily from where $A$ appears in the definition. To see $\delta(A \cup \{z_1, z_2\}) \geq \delta(A \cup \{z_1\})$, just let $B = \emptyset$, and likewise for $\delta(A \cup \{z_1, z_2\}) \geq \delta(A \cup \{z_2\})$ with $C = \emptyset$.
%we only need to put $B = \emptyset$ in \eqref{eqn:firstmadef} and $C= \emptyset$ in \eqref{eqn:secondmadef}.

For subadditivity on overlapping sets, start with the sets $A_1 \cup\{z_1\}$ and $A_2 \cup \{z_2\}$ where $A_1 \cap A_2 \neq \emptyset$.  Let $\epsilon >0$ be given.  Without loss of generality, suppose we have a set $B$  such that 
\[
\delta( A_1 \cup A_2 \cup \{z_1,z_2\}) -\epsilon \leq \delta( A_1 \cup A_2 \cup B \cup \{z_1\}) - \delta( B \cup \{z_2\}).
\]
The triangle inequality gives $\delta( A_1 \cup A_2 \cup B \cup \{z_1\}) \leq \delta(A_1 \cup \{z_1\}) + \delta( A_2 \cup B)$, and so
\[
\delta( A_1 \cup A_2 \cup \{z_1,z_2\}) -\epsilon \leq   \delta(A_1 \cup \{z_1\}) + \delta( A_2 \cup B) - \delta( B \cup \{z_2\}).
\]
The triangle inequality again gives $\delta( A_2 \cup B) \leq \delta(A_2 \cup \{z_2\}) + \delta(B \cup \{z_2\})$ and so
\[
\delta( A_1 \cup A_2 \cup \{z_1,z_2\}) -\epsilon \leq   \delta(A_1 \cup \{z_1\}) + \delta( A_2 \cup \{z_2\}).
\]
Since $\epsilon>0$ was arbitrary, we have $\delta( A_1 \cup A_2 \cup \{z_1,z_2\})  \leq   \delta(A_1 \cup \{z_1\}) + \delta( A_2 \cup \{z_2\})$.

Next,  consider the sets $A_1 \cup \{z_1\}$ and $A_2 \cup \{z_1, z_2 \}$, where $A_1$ and $A_2$ do not necessarily intersect.
Let $\epsilon >0$ be given.  Suppose there exists $B \subseteq X$ such that
\[
\delta(A_1 \cup A_2 \cup \{z_1, z_2\}) - \epsilon \leq \delta( A_1 \cup A_2 \cup B \cup \{z_1\}) - \delta(B \cup \{z_2\}).
\]
The triangle \revision{inequality} gives $\delta(A_1 \cup A_2 \cup B \cup \{z_1\}) \leq \delta(A_1 \cup \{z_1\}) + \delta(A_2 \cup B \cup \{z_1\})$ and so
\begin{eqnarray*}
\delta(A_1 \cup A_2 \cup \{z_1, z_2\}) -\epsilon & \leq & \delta(A_1 \cup \{z_1\}) + \delta(A_2 \cup B \cup \{z_1\})  - \delta(B \cup \{z_2\}) \\
& \leq & \delta(A_1 \cup \{z_1\}) + \delta( A_2 \cup \{z_1, z_2 \})
\end{eqnarray*}
where we have used the definition of $\delta( A_2 \cup \{z_1, z_2 \})$.

On the other hand, suppose there exists a $C \subseteq X$ such that 
\begin{equation*} \label{eqn:triangleAltern}
\delta(A_1 \cup A_2 \cup \{z_1, z_2\}) - \epsilon \leq \delta( A_1 \cup A_2 \cup C \cup \{z_2\}) - \delta(C \cup \{z_1\}).
\end{equation*}
Add and subtract $\delta(A_1 \cup C \cup \{z_1\})$ from the left-hand side to get
\begin{eqnarray*}
\delta(A_1 \cup A_2 \cup \{z_1, z_2\}) - \epsilon &  \leq & \delta( A_1 \cup A_2 \cup C \cup \{z_2\}) - \delta( A_1 \cup C  \cup \{z_1\}) + 
\\ & &  \delta( A_1 \cup C  \cup \{z_1\})- \delta(C \cup \{z_1\}) \\
& \leq & \delta( A_2 \cup \{z_1,z_2\}) + \delta( A_1 \cup \{z_1\}) 
\end{eqnarray*}
where we have used the definition of the amalgamation and the fact that 
\[
\delta( A_1 \cup C  \cup \{z_1\})- \delta(C \cup \{z_1\} \leq  \delta( A_1 \cup \{z_1\}) 
\]
by the triangle inequality.
\end{proof}

\begin{thm} \label{thm:findivfrai}
The class of finite diversities viewed as relational metric structures satisfies HP, JEP, bAP, PP, CP, and is therefore a \Fraisse class with an ultrahomogeneous limit.
\end{thm}
\begin{proof}
\revision{To obtain HP, observe that the diversity axioms are just equalities and inequalities that hold for all points in the diversity, so taking a subset of the diversity cannot violate any of the axioms.}

To obtain JEP, for any two finite diversities $(X,\delta_X)$ and $(Y,\delta_Y)$, let $Z$ be the disjoint union of $X$ and $Y$. Let $k$ be the largest value taken by either of the diversities $\delta_X$ and $\delta_Y$. Let $\delta_Z$ be defined as the extension of $\delta_X$ and $\delta_Y$  such that $\delta_Z(A)=k$ for any set $A$ with nonzero intersection with both $X$ and $Y$. It is straightforward to check that $(Z,\delta_Z)$ is a diversity.

To show bAP, we apply the bounded amalgamation in Def.\ \ref{defn:ma} to the two tuples $(\bar a, z_1)$ and $(\bar a, z_2)$ to get $(\bar a, z_1, z_2)$. Within this tuple,
we have that 
\begin{eqnarray*}
d(z_1,z_2)  =  \delta(\{z_1,z_2\})  
 =  \sup_{A \subseteq \bar a} | \delta( A \cup \{z_1\}) - \delta( A \cup \{z_2\}) | 
\end{eqnarray*}
giving $d(z_1,z_2) \leq d_\infty( (\bar a,z_1), (\bar a,z_2))$, as required.

To show PP:  The set of all enumerated diversities on $n$ points can be viewed as a subset of $\R^k$ for some finite $k$. This subset is closed and separable under the Euclidean metric, since it is the set of points that satisfies a family of non-strict inequalities.  So this subset of $\R^{k}$ is a Polish space.
 We just need to show that for each $n$, $d_{\mathcal{K}}$ is Lipschitz equivalent to the Euclidean metric. 
 Since the Euclidean metric is Lipschitz  equivalent to $d_\infty$, it suffices to show that $d_\mathcal{K}$ is Lipschitz equivalent to $d_\infty$ which is the content of Lemma \ref{lem:metricequiv} below.
 
 %The constants of equivalence will depend on $n$.

%For any two enumerated diversities $\bar{a}$, $\bar{b}$ we define the $\max$ metric between them as
%\[
%d_\infty (\bar{a},\bar{b}) = \max_{X \subseteq \{1,\ldots, n\}} | \delta_a( X) - \delta_b(X)|
%\]
%When viewing enumerated diversities as vectors in $\R^{2^n}$, this is Lipschitz equivalent to the Euclidean metric.  
%Lemma~\ref{lem:metricequiv} below
% shows that $d_\mathcal{K}$ is Lipschitz equivalent to $d_\infty$ and hence to the Euclidean metric. So $(\mathcal{K}_n,d_\mathcal{K})$ is a Polish space for each $n$.

To show CP: The only $n$-ary predicate on $\mathcal{K}$ is the $n$th diversity predicate $\delta^{(n)}(x_1,\ldots,x_n)= \delta(\{x_1,\ldots,x_n\})$, deleting repeated elements in the set listing.  Let $\bar a$, $\bar b$ be two diversities in $\mathcal{K}_n$. Suppose $d_{\mathcal{K}}(\bar a, \bar b) = \epsilon$. Let $(\bar a', \bar b')$ be an embedding of $\bar a$, $\bar b$ so that $\max_i (a_i',b_i') \leq 2 \epsilon$. By Lemma~\ref{lem:1-Lipschitz}, 
\[
|\delta^{(n)} (\bar a) - \delta^{(n)} (\bar b) | = |\delta'(\bar a')- \delta'(\bar b')| \leq 2 n \epsilon,
\]
showing that the map $\delta^{(n)} \colon \mathcal{K}_n \rightarrow \R$ is Lipschitz continuous.
\end{proof}

%The subset of $\R^{2^n}$ which satisfies the diversity axioms (which consist of a bunch of non-strict inequalities) is closed, so the space of diversities on $n$ points is separable and closed with respect to $d_\infty$.

\begin{lem} \label{lem:metricequiv}
For all $n$ and $\bar{a}, \bar{b}$
\[
\frac{1}{n} d_\infty(\bar{a},\bar{b}) \leq d_{\mathcal{K}}(\bar{a},\bar{b}) \leq  d_\infty (\bar{a},\bar{b}).
\]
\end{lem}
\begin{proof}
The first inequality follows from the predicates being 1-Lipschitz. In particular,
suppose we have an embedding of $\bar a$ and $\bar b$ into a third diversity $((\bar a',\bar b'),\delta)$ so that $\max_i d( a'_i, b'_i) < d_{\mathcal{K}}(\bar{a},\bar{b}) + \epsilon$. For any set $X \subseteq \{1, \ldots, n\}$
\[
|\delta_a( \bar a_X) - \delta_b(\bar b_X)| =  |\delta( \bar a'_X) - \delta(\bar b'_X)| \leq 
%\sum_{i \in X} \delta( \bar a'_i, \bar b'_i ) \leq 
\sum_{i=1}^n d( \bar a'_i, \bar b'_i ) \leq n \max_i d( \bar a'_i, \bar b'_i ) < n d_\mathcal{K}(\bar a, \bar b) +  n \epsilon,
\]
where  we have used Lemma~\ref{lem:1-Lipschitz}.
Taking the limit as $\epsilon \rightarrow 0$ and dividing by $n$ gives the first inequality.
% gives $d_\infty(\bar{a},\bar{b}) \leq n n d_\mathcal{K}(\bar a, \bar b)$. Dividing by $n$ gives the first inequality.

To prove the second inequality we will repeatedly use the bAP for diversities. Suppose $d_\infty(\bar{a},\bar{b}) =\epsilon$. We will construct a joint embedding of $\bar{a}$ and $\bar{b}$ into a new diversity where corresponding points in the tuples are never further than $\epsilon$ away from each other.

First we identify $a_1$ and $b_1$. Consider the two tuples $(a_1,b_1,a_2)$ and $(a_1,b_1,b_2)$. Since $d_\infty(\bar{a},\bar{b})=\epsilon$, we know that $d_\infty((a_1,b_1,a_2),(a_1,b_1,b_2)) \leq \epsilon$, and so we can use bAP for diversities to find a joint embedding $(a_1,b_1,a_2,b_2)$ where $d(a_2,b_2) \leq \epsilon$. We repeat this step for the tuples $(a_1,b_1,a_2,b_2,a_3)$ and $(a_1,b_1,a_2,b_2,b_3)$, and so forth. In the end we have a joint embedding of all the points in $\bar{a}$ and $\bar{b}$ such that $d(a_i, b_i) \leq \epsilon$ for all $i= 1, \ldots, k$, as required.
\end{proof}

\section{$L_1$ Metrics  and $L_1$ Diversities}

Recall that $L_1(\Omega,\mathcal{A},\mu)$ is the set of $\mathcal{A}$-measurable functions $f$ defined on $\Omega$ with $\int_\Omega |f(\omega)| d\mu(\omega) < \infty$ equipped with the metric
\[
d(f,g)= \int_\Omega |f(\omega) - g(\omega)| d\mu(\omega).
\]
An \emph{$L_1$ metric space} is a metric space that can be embedded in $L_1(\Omega,\mathcal{A},\mu)$ for some measure space $(\Omega,\mathcal{A},\mu)$.  
There is a well-developed theory of $L_1$ metric spaces, including many alternative characterizations of them in the finite case \revision{\cite[Ch. 3 \& 4]{Deza97}}. In particular, a metric on a finite set $X$ is $L_1$ if and only if it can be written as a non-negative linear combination of \emph{cut semimetrics}, $d_{U|\bar{U}}$ \revision{\cite[Thm.\ 4.2.6]{Deza97}}. \revision{To explain, letting $\bar{U} = X \setminus U$, the cut semimetric 
$d_{U|\bar{U}}$ is defined by
\[
d_{U|\bar{U}}(x,y) = \left\{
\begin{array}{cl}
1, & \mbox{if } (x \in U \mbox{ and } y \in \bar{U}) \mbox{ or } (y \in U \mbox{ and } x \in \bar{U}),\\
0, & \mbox{otherwise}.
\end{array} \right.
\]
Then $d$ is an $L_1$ metric if and only if
%\begin{equation} \label{eqn:l1charmetric}
\[
d(x,y)= \sum_{U \subseteq X} \lambda_U d_{U| \bar{U}} (x,y)
\]
%\end{equation}
for some $\lambda_U \geq 0$. }

\revision{$L_1$ diversities were introduced in \cite{Bryant14}. To define this class of diversities, we first define a  particular diversity function $\delta$ on $L_1(\Omega,\mathcal{A},\mu)$.
 %for any measure space $(\Omega,\mathcal{A},\mu)$. 
  For any finite set of functions $F$ in $L_1(\Omega,\mathcal{A},\mu)$
   we define the diversity of $F$ to be
   \begin{equation} \label{eqn:diversity_function}
\delta( F) = \int_{\Omega} \max_{f\in F} f(\omega) - \min_{g \in F} g(\omega) \,  d\omega.\end{equation}}
\revision{In \cite[p. 4]{Bryant14} we showed that this is a diversity, and if we restrict $F$ to only having two points this gives the $L_1$ metric as its induced metric. 
 Now we define 
an \emph{$L_1$ diversity} to be a diversity that can be embedded in $L_1(\Omega,\mathcal{A},\mu)$ with the diversity function $\delta$ given by \eqref{eqn:diversity_function}, for some measure space $(\Omega,\mathcal{A},\mu)$.
%where the $L_1$ diversity on that space is defined by
}

\revision{Analogous to cut semimetrics, for any partition $U|\bar{U}$ of a set $X$ there is a cut semidiversity given by \cite[p. 9]{Bryant14}
\[
\delta_{U|\bar{U}}(A) = \left\{
\begin{array}{cl}
1, & \mbox{if } (A \cap U \neq \emptyset) \mbox{ and } (A \cap \bar{U}  \neq \emptyset ),\\
0, & \mbox{otherwise}.
\end{array} \right.
\]
In \cite[Prop. 9]{Bryant14} we showed that 
a finite diversity $(X,\delta)$ is $L_1$ if and only if $\delta$ can be written as a non-negative linear combination of \emph{cut semidiversities}:
\[
\delta(A) = \sum_{U \subseteq X} \lambda_U \delta_{U|\bar{U}} (A),
\]
where all $\lambda_U \geq 0$.}
%\[
%\delta_{U|\bar{U}}(A) = \left\{
%\begin{array}{cl}
%1, & \mbox{if } (A \cap U \neq \emptyset) \mbox{ and } (A \cap \bar{U}  \neq \emptyset ),\\
%0, & \mbox{otherwise}.
%\end{array} \right.
%\]
%(Note that cut semidiversities are not actually diversities, since $\delta_{U|\bar{U}}$ can give value 0 to sets with more than 1 point.)
The induced metric of an $L_1$ diversity is $L_1$, and conversely,  every $L_1$ metric is the induced metric of some $L_1$ diversity. \revision{Both of these facts follow from the definition of both $L_1$ metrics and $L_1$ diversities in terms of embedding into $L_1(\Omega,\mathcal{A},\mu)$}.
% Conversely, . \revision{This can be seen from comparing \eqref{} and \eqref{}, and noting that $d_{U|\bar{U}}$ is the induced metric of $\delta_{U|\bar{U}}$ for any $U$.}
% \cite{Bryant14}.

In general, finite $L_1$ metrics do not have a unique decomposition in cut semimetrics \cite[\revision{Sec.\ 4.3}]{Deza97}; \revision{see \cite[Eqn.\ 9]{bryant2007linearly} for an example}. An attractive feature of $L_1$ diversities is that each $L_1$ diversity has a unique decomposition into cut semidiversities; \revision{see \cite[Prop. 10]{Bryant14} which is derived from  \cite[Thm.\ 3 \& 4]{BryantKlaere}}. \revision{Later} we show another nice feature of finite $L_1$ diversities: they satisfy the bounded Amalgamation Property (bAP) and in turn have an exact \Fraisse limit. Conversely, we show that the class of $L_1$ metric spaces does not even satisfy NAP, and hence does not have even an approximate \Fraisse limit.

\subsection{There is no Approximate \Fraisse limit for $L_1$ metrics}

We will prove that there is no \Fraisse limit (neither approximate nor exact) for finite $L_1$ metrics by showing that the class does not satisfy NAP. This can be done with a simple example.

We start by considering the simplest finite metric space that cannot be embedded in $L_1$. All metric spaces on four or fewer points can be embedded in $L_1$ \cite[p. 190]{witsenhausen1986minimum}. The following five-point metric space cannot:
\[
\begin{array}{l|ccccc} & a & b & c & z_1 & z_2\\ \hline
a & 0 & 2 & 2 & 1 & 1 \\ 
b  & & 0 & 2 & 1 & 1 \\
c  & & & 0 & 1 & 1 \\
z_1  & & & & 0 & 2 
\\
z_2  & & & & & 0
\end{array}
\]
%\andre{don't need the last row, first column}
(This is the shortest path metric in the complete bipartite graph $K_{2,3}$.)  The proof that this metric is not $L_1$ is to show that it does not satisfy the following pentagonal inequality \cite[Sec.\ 6.1]{Deza97}
\begin{equation} \label{eqn:pentagonal}
\sum_{x,y \in \{a,b,c\}} d(x,y) + \sum_{x,y \in \{z_1,z_2\}} d(x,y) - 2 \sum_{x \in \{a,b,c\}, y \in \{z_1,z_2\}} d(x,y) \leq 0.
\end{equation}

A natural way to approach finding a counter-example to the NAP is to choose two metric spaces with a common substructure that when amalgamated must yield this metric or another one violating the pentagonal inequality. Some experimentation gave the following pair of metric spaces
\[
\begin{array}{l|ccccc} & a & b & c & e & z_1\\ \hline
a & 0 & 2 & 2 & 2 & 1 \\ 
b & & 0 & 2 & 2 & 1 \\
c & & &  0 & 1 & 1 \\
e & & & & 0 & 1 \\
z_1 & & & & & 0 
\end{array}
\ \ \ \ \ \  \ \ \ \ \ 
\begin{array}{l|ccccc} & a & b & c & e & z_2\\ \hline
a & 0 & 2 & 2 & 2 & 1 \\ 
b & & 0 & 2 & 2 & 1 \\
c & & & 0 & 1 & 1 \\
e & & & & 0 & 2 \\
z_2 & & & &  & 0 
\end{array}
\]
Amalgamating these two metric spaces while identifying the common substructure on $\{a,b,c,e\}$ gives
\[
\begin{array}{l|cccccc} & a & b & c & e & z_1 & z_2\\ \hline
a & 0 & 2 & 2 & 2 & 1 & 1\\ 
b & & 0 & 2 & 2 & 1 & 1  \\
c & & & 0  & 1 & 1 & 1 \\
e & & & & 0 & 1 & 2 \\
z_1 & & & & & 0 & \gamma \\
z_2 & & & & & & 0
 \end{array}
\]
where $\gamma >0$ needs to be determined. For the amalgamation to be a valid metric space, we need $1 \leq \gamma \leq 2$. However, the pentagonal inequality is not satisfied for any value of $\gamma$ in this range. Evaluating \eqref{eqn:pentagonal} for this metric space gives  $2 \gamma \leq 0$, which cannot hold for any $\gamma$ in the range. So the class of all finite $L_1$ metrics does not satisfy AP. Hence, there is not a exact \Fraisse limit for finite $L_1$ metrics.

Is it still possible for finite $L_1$ metrics to have an approximate \Fraisse limit? To rule out this possibility we show that the set of finite $L_1$ metrics does not even satisfy NAP.  
The approximate amalgamation property requires us to embed the two metric spaces above in a common metric space, where instead of the images of $a,b,c$ being common, they just have to be arbitrarily close to each other. But the pentagonal inequality is a closed condition, and so any metric space arbitrarily close to our counter-example must also fail to satisfy it.
%fails to hold by a wide margin, so it is impossible for an arbitrarily good approximate amalgamation to also satisfy it. 
So NAP does not hold and there is no approximate \Fraisse limit for finite $L_1$ metric spaces.

\subsection{An Exact \Fraisse limit for $L_1$ diversities}

Finally, if we consider $L_1$ diversities, there is an exact \Fraisse limit, as we show here. Its induced metric is a universal $L_1$ metric, but is not ultrahomogeneous, by the results of the previous subsection.

\begin{thm} The class of finite $L_1$ diversities viewed as relational metric structures satisfies HP, JEP, bAP, PP, CP, and is therefore a \Fraisse class with a separable ultrahomogeneous limit. Furthermore, the limit is universal with respect to separable $L_1$ diversities, and is the unique such $L_1$ diversity.
\end{thm} 
\begin{proof}
HP. \revision{As in the general diversity case, taking a subset of points cannot violate any diversity axioms. Furthermore, if a diversity can be embedded in $L_1(\Omega,\mathcal{A},\mu)$ so can any subdiversity of it.}

JEP. This follows from bAP if we identify a single point in each of the finite $L_1$ diversities.

bAP.
Suppose that $(A,\delta_A)$ embeds into $(B,\delta_B)$ and into $(C,\delta_C)$. Suppose that all three diversities are $L_1$-embeddable. We need to show that $(B,\delta_B)$ and $(C,\delta_C)$ can be simultaneously embedded into a $L_1$ diversity $(D,\delta_D)$.

We assume that $B\setminus A$ is disjoint from $C\setminus A$. To help index the splits, we fix $a \in A$. Then the three diversities can be written as
\begin{align*}
\delta_A &= \sum_{U \subseteq A \setminus \{a\}, U \not = \emptyset} \alpha_{U}  \delta_{U|(A \setminus U)}\\
\delta_B &= \sum_{V \subseteq B \setminus \{a\}, V \not = \emptyset} \beta_{V}  \delta_{V|(B \setminus V)}\\
\delta_C &= \sum_{W \subseteq C \setminus \{a\}, W \not = \emptyset} \gamma_{W}  \delta_{W|(C \setminus W)}
\end{align*}
where $\alpha_U, \beta_V, \gamma_W$ are all non-negative.

We know that the three diversities all agree on subsets of $A$:
\[
\delta_A = \delta_B |_A = \delta_C |_A.
\]
$L_1$ diversities on $A$ are uniquely expressed as a weighted sum of splits of $A$. We need to figure out how to write $\delta_B|_A$ and $\delta_C |_A$ as a weighted sum of splits of $A$.
But for each split of $A$, there are many corresponding splits of $B$ (or $C$) that have the same effect on subsets of $A$. So we can break down $\delta_B|_A$ as 
\[
\delta_B = \sum_{U \subseteq A \setminus \{a\} , U \neq \emptyset} \left[ \sum_{V  \subseteq B \colon V \cap A=U} \beta_V \delta_{V | (B\setminus V)} \right],
\]
and a similar expression holds for $\delta_C$.
For each split $U \subseteq A$, $a \not \in U$, $U \not =\emptyset$, we then have 
\[\sum_{V \subseteq B: V \cap A = U } \beta_{V}  = \sum_{W \subseteq C: W \cap A = U } \gamma_{W} = \alpha_{U}.\]
%This step is justified by the unique decomposition of $L_1$-diversities property. Look at the restrictions of $\delta_B$ and $\delta_C$ to $A$. Write them as sums of splits of $A$. Then you know that corresponding coefficients are equal.

%The analogous step for metrics would not work, because there is not a unique representation of $L_1$-metrics as non-negative sums of splits.
We will now define the amalgamated $L_1$ diversity.
Let $M^{(U)}$ be a matrix with rows indexed by elements of $\{V \subseteq B: V \cap A = U\}$; columns indexed by elements of $\{W \subseteq C: W \cap A = U\}$; such that for all $V$
\[\sum_{W \subseteq C: W \cap A = U } M^{(U)}_{VW} = \beta_{V} \]
and for all $W$,
\[\sum_{V \subseteq B: V \cap A = U } M^{(U)}_{VW} = \gamma_W.\]
Any such matrix provides an amalgamation
 $\delta_D$ given by
\[\delta_D = \sum_{U \subseteq A \setminus \{a\}, U \neq \emptyset} \sum_{V \subseteq B: V \cap A = U } \sum_{W \subseteq C: W \cap A = U } M^{(U)}_{VW} \delta_{V \cup W | ((B \cup C) \setminus (V \cup W))}.\]

%One such matrix is given by 
%\[M_{VW}^{(U)} = \frac{\beta_V \gamma_W}{\alpha_U}.\]
%This will not satisfy bAP.

We need to find a choice of $M^{(U)}$ such that bAP holds. 
 We just consider two-point amalgamation; the general case follows by induction.
We let $B=A \cup\{z_1\}$ and $C = A \cup \{z_2\}$.
Now $U$ runs over all nonempty subsets of $A$ not containing $\{a\}$. But for each $U$,
$V$ just takes the values $U$ and  $U \cup \{z_1\}$ %for $A$ not containing $\{a\}$.
and 
$W$ just takes the values $U$ and  $U \cup \{z_2\}$. %for $A$ not containing $\{a\}$.

Now our amalgamated diversity simplifies to (only writing the one half of the splits in the split notation, i.e.\ $U$ for $U |( (A \cup \{z_1,z_2\})\setminus U$)
\[
\delta_D = \sum_{U \subseteq A \setminus \{a\}} 
M^{(U)}_{U,U} \delta_{U}  %| (B\cup C) \setminus U} 
+
M^{(U)}_{U\cup \{z_1\},U} \delta_{U \cup \{z_1\}} % | (B\cup C) \setminus (U \cup \{z_1\})}
+
 M^{(U)}_{U,U \cup \{z_2\}} \delta_{U \cup \{z_2\}}  % | (B\cup C) \setminus (U \cup \{z_1\})}
+
  M^{(U)}_{U \cup\{z_1\},U\cup \{z_2\}}\delta_{U \cup \{z_1,z_2\}} % | (B\cup C) \setminus (U \cup \{z_1\})}
\]
If we only want to know the value of the diversity on $\{z_1,z_2\}$ then it simplifies to 
\begin{equation} \label{eqn:twopointL1}
\delta_D(\{z_1,z_2\})   =   \sum_{U \subseteq A \setminus \{a\}} 
M^{(U)}_{U\cup\{z_1\},U} 
+
M^{(U )}_{U,U \cup \{z_2\}}
\end{equation}

We choose the entries of $M^{(U)}$ analogously to how we chose to amalgamate stochastic processes in Section~\ref{sec:stochproc}.
We let
\begin{eqnarray*}
M_{U,U}^{(U)} & = & \min(\beta_U,\gamma_U) \\
M_{U\cup\{z_1\},U\cup\{z_2\}}^{(U)} & = & \min( \beta_{U \cup \{z_1\}}, \gamma_{U \cup \{z_2\}}) \\
M_{U\cup\{z_1\},U}^{(U)} & = &%\frac{[ \beta_{U \cup \{z_1\}} - \min( \beta_{U \cup \{z_1\}} , \gamma_{U \cup \{z_2\}})]
%[
 \gamma_U - \min( \beta_U, \gamma_U )
 %]
%}{  ?   }
 \\
M_{U,U\cup\{z_2\}}^{(U)} & = &  
%\frac{[ \beta_{U} - 
\beta_U  - \min( \beta_{U} , \gamma_{U})
%]
%[ \gamma_{U \cup \{z_2\}} -
 %\min( \beta_{U\cup\{z_1\}}, \gamma_{U \cup \{z_2\}} )
 %]
 %}{ ?    }   
%M_{VW}^{(U)} & = & \delta_{VW} m_V + \frac{ (\beta_V - m_V ) (\gamma_W - m_W)}{\Gamma_U}
\end{eqnarray*}
%where $m_V = \min( \beta_V, \gamma_V)$ and $\Gamma_U = \sum_{V} (\gamma_V - m_V)$. 
for each $U \subseteq A, a \not \in U$. 
%These are the only $V$ and $W$ we need to specify.

Plugging these choices into \eqref{eqn:twopointL1} gives
\begin{eqnarray*}
\delta_D(\{z_1,z_2\}) = \sum_{U \subseteq A \setminus \{a\}} \beta_U + \gamma_U -  2 \min( \beta_U, \gamma_U) =   \sum_{U \subseteq A \setminus \{a\}}  |\beta_U- \gamma_U|.
\end{eqnarray*}
We now have to bound this expression  in terms of differences of $\delta_B(U \cup \{z_1\})$ and $\delta_C(U \cup \{z_2\})$. To do this we need to express $\beta_U$ and $\gamma_U$ in terms of the diversities evaluated on set. 
\revision{Equation (7) in \cite[p. 12]{Bryant14} gives} the expression for the weights when $\lambda = \bar{\lambda}$. We are only taking splits $U|\bar{U}$ where $U$ doesn't contain $A$. So we have to double the weight in that paper.
We get, for all $U \subseteq A$, $a \not \in U$:
%\[
%\alpha_U = \sum_{V: U\subseteq V} (-1)^{|U|-|V|+1} \delta(V)
%\]
\[
\beta_U =  \sum_{V: U\subseteq V, V \subseteq A} (-1)^{|U|-|V|+1} (\delta(V)- \delta(V \cup\{z_1\}))
\]
%\[
%\beta_{U \cup \{z_1\}} = \sum_{V: U\subseteq V, V \subseteq A} (-1)^{|U|-|V|} \delta(V \cup \{z_1\})
%\]
\[
\gamma_U = \sum_{V: U\subseteq V, V \subseteq A} (-1)^{|U|-|V|+1} (\delta(V)- \delta(V \cup\{z_2\}))
\]
%\[
%\gamma_{U \cup \{z_2\} }= \sum_{V: U\subseteq V, V \subseteq A} (-1)^{|U|-|V|} \delta(V \cup \{z_2\})
%\]
%
%Multiplying
%\begin{eqnarray*}
%\beta_U \gamma_{U \cup \{z_2\}} &=& \sum_{V_1, V_2 : U \subseteq V_1, V_2} 
%(-1)^{-|V_1|-|V_2|+1} \left( \delta( V_1) - \delta(V_1 \cup \{z_1\}) \right) \delta(V_2 \cup \{z_2\}) \\
%\beta_{U\cup \{z_1\}} \gamma_{U} &=& \sum_{V_1, V_2 : U \subseteq V_1, V_2} 
%(-1)^{-|V_1|-|V_2|+1} \left( \delta( V_2) - \delta(V_2 \cup \{z_2\}) \right) \delta(V_1 \cup \{z_1\})  
%\end{eqnarray*}
This gives us
\begin{eqnarray*}
\delta_D(\{z_1,z_2\}) & = & \sum_{U \subseteq A \setminus \{a\}} \left| \sum_{V: U\subseteq V, V \subseteq A} (-1)^{|U|-|V|+1} (\delta(V \cup \{z_1\})- \delta(V \cup\{z_2\})) \right| \\
& \leq & \sum_{U \subseteq A \setminus \{a\}} \sum_{V: U\subseteq V, V \subseteq A} \left|\delta(V \cup \{z_1\})- \delta(V \cup\{z_2\}) \right| \\
& \leq & 2^{2n} \max_{V \subseteq A} \left|\delta(V \cup \{z_1\})- \delta(V \cup\{z_2\}) \right|
%\leq 2^{2n} d_\infty
\end{eqnarray*}
as required.

PP.  We follow the exact same argument as in Theorem~\ref{thm:findivfrai}, \revision{noting that $L_1$ diversities also can be viewed as a closed subset of $\mathbb{R}^k$ for large enough $k$.}
\revision{We only need to establish the analogue of Lemma~\ref{lem:metricequiv} for $L_1$ diversities. The inequality $d_\infty(\bar a,\bar b) \leq n d_{\mathcal{K}}(\bar a,\bar b)$ follows exactly as in Lemma~\ref{lem:metricequiv}. To get a bound on $d_\mathcal{K}(\bar a,\bar b)$ in terms of $d_\infty(\bar a,\bar b)$ we use bAP for $L_1$ diversities repeatedly. }

\revision{First, suppose that $d_\infty(\bar a,\bar b)=\epsilon$. We will construct a joint  embedding of $\bar a$ and $\bar b$ into a new $L_1$ diversity where corresponding points in the tuples are never further than $2^{2n} \epsilon$ from each other. } 

\revision{First we identify $a_1$ and $b_1$. Consider the two tuples $(a_1,b_1,a_2)$, $(a_1,b_1,b_2)$. Since $d_\infty(\bar a,\bar b)=\epsilon$, we know that $d_\infty((a_1,b_1,a_2),(a_1,b_1,b_2)) \leq \epsilon$, and so be can use bAP for $L_1$ diversities to obtain a joint embedding $(a_1,b_1,a_2,b_2)$ where $d(a_2,b_2)\leq 2^{2n} \epsilon$. We repeat this step for the tuples $(a_1,b_1,a_2,b_2,a_3)$ and $(a_1,b_1,a_2,b_2,b_3)$, and so forth. In the end we have a joint embedding of all the points in $\bar{a}$ and $\bar{b}$ such that $d(a_i, b_i) \leq 2^{2n} \epsilon$ for all $i= 1, \ldots, k$.
   }
 \revision{  Thus we have obtained
\[
\frac{1}{n} d_\infty(\bar{a},\bar{b}) \leq d_{\mathcal{K}}(\bar{a},\bar{b}) \leq  2^{2n} d_\infty (\bar{a},\bar{b}),
\]
as required.}

CP. This follows exactly like the proof of the property CP in Theorem~\ref{thm:findivfrai}, the restriction to $L_1$ diversities not making any difference.

\end{proof}

{\bf Acknowledgment.}  DB was supported by a University of Otago research grant. AN was supported by the Marsden fund of New Zealand. PT was supported by an NSERC Discovery Grant and a Tier 2 Canada Research Chair.

\bibliographystyle{asl}
%\bibliography{BryantNiesTupper.bib}

\end{document}